\documentclass[11pt]{article}
\usepackage{geometry}
\usepackage{color}
\usepackage{xcolor}
\usepackage{float}
\usepackage{textcomp}
\usepackage{amsthm}
\usepackage{amsmath}
\usepackage{amssymb}%
\usepackage{multirow}
\usepackage{changes}
\usepackage{url}
\usepackage{hyperref}
\usepackage{xspace}

\usepackage{mathtools}
\usepackage{booktabs}    %pacote estético para tabelas
\usepackage{subfigure} 
\usepackage{caption} 
\usepackage{multirow}
\usepackage{multicol}    %para ajustar alinhamentos de colunas
\usepackage{array}       %para ajustar espaçamentos
\usepackage{graphicx}
\graphicspath{{figuras/}}
\usepackage{empheq,amsmath}
\usepackage[framemethod=tikz]{mdframed}
\usepackage
[ %linesnumbered, 
lined, 
boxruled, %more options => boxed  algoruled  tworuled  plain
commentsnumbered
]
{algorithm2e}
\DecMargin{0.1cm} %diminui a margem esquerda do algorithm
\newtheorem{theorem}{Theorem}[section]
\newtheorem{lemma}[theorem]{Lemma}
\newtheorem{corollary}[theorem]{Corollary}
\newtheorem{proposition}[theorem]{Proposition}

\newcommand{\R}{\mathbb{R}}

\newcommand{\bi}{\begin{itemize}}
\newcommand{\ei}{\end{itemize}}
\newcommand{\ba}{\begin{array}}
\newcommand{\ea}{\end{array}}

\begin{document}

\title{A Regularized Hessian-Free Inexact Newton-Type Method with Global $\mathcal{O}(k^{-2})$ Convergence}

% The thanks line in the title should be filled in if there is
% any support acknowledgement for the overall work to be included
% This \thanks is also used for the received by date info, but
% authors are not expected to provide this.

\author{Leandro Farias Maia\thanks{School of Mechanical, Industrial, and Manufacturing Engineering, Oregon State University,
Corvallis, OR, 97331. \protect\protect\href{mailto:fariasmaia@gmail.com}{fariasmaia@gmail.com}}\hspace*{0.5em} \and Antônio Victor B. Nascimento\thanks{Catholic Pontifical University,
Rio de janeiro, RJ, 22451-900.  
\protect\protect\href{mailto:victornascimento@aluno.puc-rio.br}{victornascimento@aluno.puc-rio.br}} \and Paulo Sérgio M. Santos\thanks{Federal University of Delta do Parnaíba, Parnaíba, PI, 64202-020. 
\protect\protect\href{mailto:paulo.santos@ufdpar.edu.br}{paulo.santos@ufdpar.edu.br}}\and Gilson N. 
Silva\thanks{Department of Mathematics, Federal University of Piauí, Teresina, PI, 64049-550. \protect\protect\href{mailto:monteiro@isye.gatech.edu}{gilson.silva@ufpi.edu.br}}}

\maketitle

\begin{abstract}

We propose a regularized Hessian-free Newton-type method for minimizing smooth convex functions with Lipschitz continuous Hessians. 
The algorithm constructs an approximate Hessian by finite differences and selects the regularization parameter through an adaptive criterion that ensures sufficient decrease and gradient control. 
We prove that the method achieves a global $\mathcal{O}(k^{-2})$ convergence rate, matching the best known bound for second-order methods. 
A modified variant incorporating the exact Hessian when available enjoys local quadratic convergence under standard assumptions. 
Despite its simplicity, this variant is computationally faster than the \emph{Regularized Newton Method} of Mishchenko (2023) across several convex benchmark problems. 
Our analysis also provides explicit bounds on the regularization sequence and a worst-case iteration complexity of order $\mathcal{O}(\varepsilon^{-2})$. 
The proposed framework thus unifies regularized and Hessian-free Newton-type schemes, offering a theoretically sound and practically efficient alternative for smooth convex optimization.

\vspace{0.5em}
\noindent{\bf keywords} Newton-type method, quadratic convergence, complexity, Hessian-free

\end{abstract}

\section{Introduction}

The minimization of a smooth convex function $f:\mathbb{R}^n \to \mathbb{R}$ is a fundamental problem in continuous optimization, with applications spanning from machine learning to structural engineering. 
In high-dimensional settings, the efficiency of an algorithm is often measured by its oracle complexity—the number of calls to a first or second-order oracle required to reach an $\varepsilon$-approximate solution. 
While first-order methods are widely used due to their low per-iteration cost, 
second-order methods, such as Newton's method, 
% are preferred when high precision is required or when the objective function is ill-conditioned. However, the classical Newton method often lacks global convergence guarantees without restrictive assumptions or sophisticated line-search strategies, which may fail in practice due to numerical instabilities [38].
% Our goal in this work is 
% Second-order methods 
constitute one of the most powerful classes of algorithms for smooth optimization problems. 
By exploiting curvature information through the Hessian matrix, Newton-type methods are capable of achieving 
fast local convergence rates, often quadratic under standard regularity assumptions. 
Because of this property, Newton’s method and its variants play a central role in nonlinear optimization, 
scientific computing, machine learning, and inverse problems.

Despite its remarkable local behavior, the classical Newton method is well known to suffer from poor global 
convergence properties. 
In particular, the method may diverge when initialized far from a solution or when the 
Hessian matrix is ill-conditioned or indefinite. 
These limitations have motivated the development of several 
globalization strategies designed to stabilize Newton-type iterations while preserving their fast convergence 
near a solution \cite{grippo1986nonmonotone}.

The most common approaches to globalizing Newton’s method include line-search and trust-region techniques. 
 Line-search methods compute a Newton direction and subsequently determine an appropriate step length that 
 guarantees sufficient decrease in the objective function. Trust-region methods, on the other hand, compute 
 steps by minimizing a quadratic model within a region where the model is assumed to be reliable. 
 These approaches are widely used in practice and are extensively studied in the literature 
 (see, e.g., (see, e.g., \cite{nocedal2006numerical, conn2000trust}). 
 However, their theoretical guarantees typically lead to worst-case iteration complexities comparable 
 to first-order methods, which limits their advantage from a complexity-theoretic perspective.

A major theoretical breakthrough was achieved with the introduction of cubic regularization methods \cite{griewank1981modification}. 
 The cubic Newton method of Nesterov and Polyak \cite{nesterov2006cubic} establishes optimal worst-case complexity guarantees for 
 second-order methods and enjoys strong global convergence properties \cite{nesterov2008accelerating, cartis2011adaptive}. 
% More precisely, cubic regularization achieves an iteration complexity of order $\mathcal{O}(k^{-2})$ ???????
 %for convex problems with Lipschitz continuous Hessians and provides strong guarantees even in nonconvex settings \cite{cartis2013evaluation}. 
 However, the main drawback of cubic Newton methods lies in the computational cost of each iteration, since the 
 method requires solving a cubic subproblem that may be expensive in large-scale applications \cite{nesterov2021superfast}. See also \cite{cartis2013evaluation}.

Motivated by these limitations, several recent works have explored alternative regularization strategies that 
 maintain the favorable convergence properties of cubic regularization while preserving simpler update formulas \cite{birgin2017use, cartis2019universal}. 
 One particularly influential direction is based on Levenberg–Marquardt type regularization \cite{fan2005quadratic, dan2002convergence}, in which the Hessian 
 matrix is stabilized by adding a multiple of the identity matrix \cite{li2004regularized, polyak2009regularized, ueda2014regularized}. 
 In this framework, Newton-type iterations are obtained by solving linear systems of the form
 \[
 (\nabla^2 f(x_k)+\lambda_k {\cal I})s_k=-\nabla f(x_k),
 \]
 where the regularization parameter $\lambda_k>0$ controls the conditioning of the system and improves the 
robustness of the method and ${\cal I}$ represents the identity matrix.

 Recently, Mishchenko proposed a regularized Newton method in which the regularization parameter depends on the 
 gradient norm \cite{mishchenko2023regularized}. 
 This method provides a simple closed-form update rule and achieves a global convergence rate of order 
 $\mathcal{O}(1/k^2)$ for convex functions with Lipschitz continuous Hessians, matching the theoretical rate of 
 cubic Newton methods while avoiding the need to solve cubic subproblems. 
Moreover, the method enjoys superlinear convergence in the strongly convex setting \cite{doikov2022high, doikov2024super}. 
 These results demonstrate that carefully designed regularization mechanisms can yield efficient Newton-type 
 methods with strong theoretical guarantees \cite{alvarez2025first}.

 Despite these advances, many practical optimization problems remain challenging because computing the exact 
 Hessian matrix is either expensive or infeasible. This is particularly true in large-scale problems arising in 
 machine learning \cite{martens2010deep}, inverse problems, and scientific computing. 
 In such settings, Hessian-free approaches and quasi-Newton techniques are often preferred, since they avoid the 
 explicit construction of the Hessian matrix while still exploiting second-order information.

Finite-difference approximations constitute a natural way to construct Hessian-free models \cite{grapiglia2022cubic, grapiglia2024worst}. 
In these approaches, curvature information is obtained by evaluating gradients at nearby points, leading to 
matrix approximations that can be computed without explicit second derivatives. 
 However, combining such approximations with globally convergent Newton-type schemes remains a delicate task, 
 as approximation errors may affect both stability and convergence guarantees. 
 %The goal of this work is to develop a Newton-type framework that combines three desirable features: 
% (i) global convergence guarantees comparable to those of cubic regularization methods, 
% (ii) computational simplicity similar to Levenberg–Marquardt regularization, and 
% (iii) a Hessian-free implementation based on finite-difference approximations.

In this paper, we propose a regularized Hessian-free inexact Newton-type method that unifies the benefits of simple quadratic regularization with the practical flexibility of Hessian-free updates. Our main contributions are summarized as follows:
\begin{itemize}
    \item {\bf Hessian-Free Framework:} We utilize a finite-difference approximation of the Hessian, eliminating the need for exact second-order derivatives;
    %while preserving the method's curvature-capturing capabilities;

    \item {\bf Global Convergence:} We prove that the proposed method also achieves a global convergence rate of $
\mathcal{O}(k^{-2})
$ for convex functions with Lipschitz continuous Hessians, matching the state-of-the-art complexity for second-order methods [1, 54];

    \item {\bf Inexact Subproblem Solves:} Unlike many theoretical models, our algorithm allows for inexact solutions of the Newton-type system, making it suitable for large-scale problems where iterative linear solvers (like Conjugate Gradient) are employed;

    \item {\bf Adaptive Regularization:} %We employ an adaptive regularization mechanism that ensures both descent of the objective function and control 
 %of the gradient norm at each iteration. 
 The regularization parameters are automatically adjusted by the algorithm, guaranteeing that the theoretical 
 conditions required for convergence are satisfied without requiring user-specified constants;

    \item {\bf Numerical Efficiency:} We demonstrate 
    %through extensive benchmarking—including datasets such as mushrooms and w8a—
    that our approach outperforms recent methods such as the Regularized Newton Method [8] in terms of CPU time and total iterations.
\end{itemize}

 The remainder of the paper is organized as follows. Section~2 presents the basic assumptions and preliminary 
 results used throughout the analysis. Section~3 studies the regularized quasi-Newton step and establishes 
 fundamental descent properties. Section~4 introduces the proposed algorithm and derives bounds for the 
 regularization parameters. Section~5 provides the global and local convergence analysis of the method.

\section{Preliminary Results}
The goal of this section is to present the main problem of interest in this paper and the assumptions related to it.

We recall that the main of this paper is to solve the following unconstrained optimization problem:
\begin{equation}\label{eq:principal}
\min_{x\in \mathbb{R}^n} f(x),
\end{equation}
where $f:\mathbb{R}^n\to \mathbb{R}$ satisfies the following  assumptions:\\
%\begin{itemize}
{\bf Assumption A:} $f$ is a convex, twice differentiable function and there exists $H>0$ such that%$its Hessian  $\nabla^2 f$ is  Lipschitz continuous with constant $H$, i.e.,  for every $j=1, \ldots, d,$ $x\in \mathbb{R}^n$ and $h\in \mathbb{R}^{n_j}$
\begin{equation*}
\|\nabla f(y)-\nabla f(x)-\nabla^2f(x)(y-x)\|\leq H\|y-x\|^2.
\end{equation*}
%\item [\textbf{A3}] 
% and
% \begin{equation*}
%     \|\nabla f(x)\|\le M
% \end{equation*}
% for some $M>0.$\\
\noindent
{\bf Assumption B:} There exists $f^{\text{low}}\in \mathbb{R}$ such that $f(x)\ge f^{\text{low}}$ for each $x\in \mathbb{R}^n.$\\
%{\bf Assumption C:} there exist $\kappa_B\ge 0$ and a matrix $B\in \mathbb{R}^{n\times n},$ an approximation to $\nabla^2 f,$ such that
%\end{itemize}
{\bf Assumption C:}  For every $x_1\in \R^n$, the sublevel set
\[
{\cal L}(x_1):=\{x \in \R^n~: ~ f(x)\le f(x_1)\}
\]
is bounded.

It is a consequence of Assumption A that the following key inequality holds:
\begin{equation*}
	f(y)\le f(x)+\langle \nabla f(x), y-x\rangle+\dfrac{1}{2}\langle \nabla^{2}f(x)(y-x),(y-x)\rangle+\dfrac{H}{3}\|y-x\|^{3}.
	\label{eq:2.4}
\end{equation*}

Moreover, Assumptions A and C imply that there exists
a constant $M>0$ such that
\begin{equation}\label{M:grad-bounded}
\|\nabla f(x)\|\le M, \qquad \forall x\in {\cal L}(x_1).
\end{equation}

\section{ A regularized inexact quasi-Newton method}
In this section, we study a regularized inexact quasi-Newton method where the displacement
\begin{equation}\label{def:s} 
s := x^+ - x.
\end{equation}
satisfies the following residual condition:
% \begin{equation}\label{inexact}
%     x^+ := x - (B+\lambda{\cal I})^{-1}\nabla f(x),
% \end{equation}
\begin{equation}\label{inexact}
    \|(B+\lambda{\cal I}) s + \nabla f(x)\| \le \theta \min\{\|\nabla f(x)\|, \|s\|\}, 
    \quad \theta \in [0,1).
\end{equation}
Here, $B$ is a symmetric approximation of the Hessian $\nabla^2 f(x)$ and $\lambda>0$ is a regularization parameter.

The relations \eqref{inexact} and \eqref{def:s} define a broad family of stabilized Newton-type updates, able to control the local curvature and guarantee global regularity even when $B$ is only an inexact model of $\nabla^2 f(x)$.  
% Throughout the section, we frequently use the shorthand 

% For intuition, notice that the ideal step would satisfy
% \[
%     (B+\lambda {\cal I})s + \nabla f(x) = 0,
% \]
% and we will later consider approximate residual conditions of the type

% to discuss practical implementability.

To control the deviation of $B$ from the true Hessian, we impose the following structural assumption.

\paragraph{Assumption D.}  
There exist $\kappa_B \ge 0$ and a symmetric matrix $B \in \mathbb{R}^{n\times n}$ such that
\begin{equation}\label{cond:matrix:B}
    \|B - \nabla^2 f(x)\|
    \;\le\;
    \kappa_B \sqrt{\|\nabla f(x)\|^\alpha}, %\quad \textcolor{blue}{\|B - \nabla^2 f(x)\|\leq \kappa_B\min\{\sqrt{\sigma \|\nabla f(x)\|^\alpha}, \sigma\|\nabla f(x)\|\}},
\end{equation}
for some $\alpha\in (0,1].$
\noindent

\begin{lemma}\label{min:autovalor}
    Let $f$ be a convex function and $B$ a matrix that satisfies Assumption D. Then the following statements hold:
    \begin{itemize}
        \item[(a)] there holds
\begin{equation}\label{lem:min:autovalor}
\lambda_{\min}(B) \geq - \kappa_B \sqrt{\|\nabla f(x)\|^\alpha},
    \end{equation}
where $\lambda_{\min}(B)$ denotes the minimum eigenvalue of $B$.
    \item[(b)] for any $s \in \mathbb{R}^n$
    \begin{equation}\label{Rayleigh}
\langle Bs,s\rangle
\;\ge\;
-\kappa_B \sqrt{\|\nabla f(x)\|^\alpha}\,\|s\|^2.
\end{equation}
    \end{itemize}
\end{lemma}

\begin{proof}
    (a) Using the convexity of $f$ and \eqref{cond:matrix:B}, we have
    \begin{align*}
        \lambda_{\min}(B) &=  \lambda_{\min}(\nabla^2 f(x) +(B-\nabla^2 f(x))\\& \geq \lambda_{\min}(\nabla^2f(x)) - \|B-\nabla^2f(x)\| \\ &\geq - \|B-\nabla^2f(x)\|  \geq - \kappa_B \sqrt{\|\nabla f(x)\|^\alpha},
    \end{align*}
    which shows that \eqref{lem:min:autovalor} holds.
    
(b) For any vector $s \in \mathbb{R}^n$,
the Rayleigh quotient characterization of eigenvalues yields
\[
\langle Bs,s\rangle \;\ge\; \lambda_{\min}(B)\,\|s\|^2.
\]
The conclusion now follows by
combining the previous inequality with \eqref{lem:min:autovalor}.
\end{proof}

This error bound prescribes that as the gradient becomes small, $B$ must approximate the Hessian more accurately, which is consistent with typical finite-difference or quasi-Newton constructions.  

We begin by establishing two fundamental bounds for the step length and the gradient at the new iterate.

% \begin{corollary}[Resolvent bound]\label{cor:resolvent-bound}
% Under the assumptions of Lemma~\ref{lem:Hess-B-bounded}, assume that $\lambda\ge 2\|B\|$. Then, 
% \[
% B+\lambda I\succ 0, \quad\|(B+\lambda I)^{-1}\|\le \frac{2}{\lambda}.
% \]
% Consequently, for $s:=-(B+\lambda I)^{-1}\nabla f(x)$,
% \[
% \|s\|\le \frac{2}{\lambda}\|\nabla f(x)\|.
% \]
% \end{corollary}

% \begin{proof}
% Since $B$ is symmetric, $\lambda_{\min}(B)\ge -\|B\|$. 
% Hence
% \[
% \lambda_{\min}(B+\lambda I)=\lambda+\lambda_{\min}(B)\ge \lambda-\|B\|\ge \frac{\lambda}{2},
% \]
% which implies $\|(B+\lambda I)^{-1}\|=1/\lambda_{\min}(B+\lambda I)\le 2/\lambda$.
% The bound on $\|s\|$ follows immediately.
% \end{proof}

\begin{lemma}\label{lem:3.1}
Assume that Assumptions {\bf A} and {\bf D} hold, $\alpha \in (0,1]$, and that
\begin{equation}\label{cond:lambda:1}
   \sigma > 4 \kappa_B^2, \quad \zeta>2, \quad \lambda \geq\max\{ [2(1+\theta)]^{\frac{\alpha}{2}}\sqrt{\sigma \|\nabla f(x)\|^\alpha}, \zeta\theta\}. 
\end{equation}
Then,
\begin{equation}\label{eq:bousk1}
   \|s\|\leq
        \frac{2(1+\theta)\|\nabla f(x)\|}{\lambda}, \quad  \sigma^{\frac{1}{\alpha}}\|s\| 
    \le \lambda^{\frac{2-\alpha}{\alpha}},
\end{equation}
where $s$ is as in \eqref{def:s}. 
Moreover, if $g:=\|\nabla f(x)\|$, then
\begin{equation}\label{bound:grad}
    \|\nabla f(x^+)\|
    \le
    \left( 
        \frac{H g^{1-\alpha}}{\sigma}
        +\frac{\kappa_B}{\sqrt{\sigma}}
        +\frac{\zeta+1}{\zeta}\right)
    \lambda \|s\|,
\end{equation}
where $x^+$ and $H$ are as in \eqref{def:s} and Assumption A, respectively. In particular, if $\sigma$ satisfies
\begin{equation}\label{cond:sigma:grad}
  \sigma\ge  \left(
\frac{\zeta}{2(\zeta-1)}
\left[
\kappa_B+\sqrt{\kappa_B^2+\frac{4(\zeta-1)}{\zeta}H g^{1-\alpha}}
\right]
\right)^2:=\hat\sigma_1(g),
\end{equation}
then
\[
\|\nabla f(x^+)\|\le 2\lambda\|s\|.
\]
\end{lemma}
\begin{proof}
We first prove \eqref{eq:bousk1}. 
We begin by estimating a lower bound to  $\lambda_{\min}(B+\lambda {\cal I})$. 
It follows from \eqref{lem:min:autovalor} that 
% Since $B$ is symmetric, $\lambda_{\min}(B)\ge -\|B\|$. 
% Hence
\[
\lambda_{\min}(B+\lambda {\cal I})=\lambda+\lambda_{\min}(B)\ge \lambda- \kappa_B \sqrt{\|\nabla f(x)\|^\alpha}\ge \frac{\lambda}{2},
\]
where the last inequality above is due to the first and third inequalities in \eqref{cond:lambda:1}.
% which implies $\|(B+\lambda I)^{-1}\|=1/\lambda_{\min}(B+\lambda I)\le 2/\lambda$.
Next, let $r:=(B+\lambda{\cal I}) s + \nabla f(x)$. Then, the previous inequality, \eqref{inexact}, \eqref{cond:lambda:1} and the definition of $r$ imply that  
% \begin{align}\label{expre;autovalor}
%     \lambda_{\min}(B+\lambda \mathcal{I})
%     &\geq \lambda_{\min}(B) + \lambda
%     \overset{\eqref{lem:min:autovalor}}{\geq}
%     -\kappa_B \sqrt{\|\nabla f(x)\|^\alpha} + \lambda \nonumber\\
%     &\overset{\eqref{cond:lambda:1}}{=}
%     -\kappa_B \sqrt{\|\nabla f(x)\|^\alpha}
%     + \sqrt{\sigma \|\nabla f(x)\|^\alpha}
%     \overset{\eqref{cond:lambda:1}}{\geq}
%     \frac{1}{2}\sqrt{\sigma \|\nabla f(x)\|^\alpha}
%     \overset{\eqref{cond:lambda:1}}{=}
%     \frac{\lambda}{2}.
% \end{align}
\begin{align*}
    \|s\|
    \overset{\eqref{inexact}}{=}
    \| (B+\lambda \mathcal{I})^{-1}(\nabla f(x)-r) \|\le  \| (B+\lambda \mathcal{I})^{-1}\|(\|\nabla f(x)\|+\|r\|)\leq
    \frac{2(1+\theta)}{\lambda}\|\nabla f(x)\|,
\end{align*}
which implies that the first inequality in \eqref{eq:bousk1} holds. 
On the other hand, using the previous bound and \eqref{cond:lambda:1}, we have
\begin{align*}
\|s\|&\le\frac{2(1+\theta)}{\lambda}\|\nabla f(x)\|\overset{\eqref{cond:lambda:1}}{\leq}
    \frac{2(1+\theta)}{\lambda}\frac{1}{2(1+\theta)}\left(\frac{\lambda^2}{\sigma}\right)^{1/\alpha}
    = \frac{\lambda^{\frac{2-\alpha}{\alpha}}}{\sigma^{\frac{1}{\alpha}}},
\end{align*}
which is the second inequality in \eqref{eq:bousk1}.

We now show that \eqref{bound:grad} also holds.
Using \eqref{inexact}, the first inequality in Assumption~{\bf A}, Assumption~{\bf C},
\eqref{cond:lambda:1}, and \eqref{eq:bousk1}, we obtain
\begin{align}\label{pre:grad:f:+}
    \|\nabla f(x^+)\|
    &\overset{\eqref{inexact}}{=}
    \|\nabla f(x^+) - \nabla f(x) - Bs - \lambda s+r\| \nonumber\\
    &\le
    \|\nabla f(x^+) - \nabla f(x) - \nabla^2 f(x)s\|
    + \|(\nabla^2 f(x)-B)s\|
    + \lambda\|s\|+\|r\| \nonumber\\
    &\overset{\eqref{cond:matrix:B}}{\le}
    H\|s\|^2
    + \kappa_B \sqrt{\|\nabla f(x)\|^\alpha}\|s\|
    + (\lambda+\theta)\|s\|\nonumber \\
    &= H\|s\|^2
    + \frac{\kappa_B}{\sqrt{\sigma}}\sqrt{\sigma\|\nabla f(x)\|^\alpha}\|s\|
    + (\lambda+\theta)\|s\|\nonumber\\
    &\overset{\eqref{eq:bousk1},\eqref{cond:lambda:1}}{\le}
    \frac{H}{\sigma^{\frac{1}{\alpha}}}\lambda^{\frac{2-\alpha}{\alpha}}\|s\|
    + \frac{\kappa_B}{\sqrt{\sigma}}\lambda\|s\|
    + \left(\lambda+\frac{\lambda}{\zeta}\right)\|s\| \nonumber\\
    &\overset{\eqref{cond:lambda:1}}{=}
    \frac{H}{\sigma}g^{1-\alpha}\lambda\|s\|
    + \frac{\kappa_B}{\sqrt{\sigma}}\lambda\|s\|
    + \frac{\zeta+1}{\zeta}\lambda\|s\| \nonumber\\
    &\le \left(
        \frac{Hg^{1-\alpha}}{\sigma}
        + \frac{\kappa_B}{\sqrt{\sigma}}
        + \frac{\zeta+1}{\zeta}   \right)\lambda\|s\|.
\end{align}
% On the other hand, it follows from the definition of $r$ that
% \[
% \|\nabla f(x)\|
% \le \|B+\lambda {\cal I}\|\,\|s\|+\|r\|
% \le (\|B\|+\lambda)\|s\|+\theta\|\nabla f(x)\|,
% \]
% which implies that
% \[
% \|\nabla f(x)\|
% \le
% \frac{\|B\|+\lambda}{1-\theta}\|s\|=\frac{\left(\frac{\|B\|}{\lambda}+1\right)}{1-\theta}\lambda\|s\|\le\frac{3}{2(1-\theta)}\lambda\|s\|,
% \]
% where the last inequality above follows from \eqref{cond:lambda:1}.

Hence, we conclude that \eqref{bound:grad} holds.
%follows from the previous inequality and \eqref{pre:grad:f:+}.
The last conclusion of the lemma is an immediate consequence of \eqref{bound:grad} and \eqref{bound:grad}.
\end{proof}

The previous lemma shows that 
%the magnitude of the step, i.e.,  
$\|s\|$ is controlled by the regularization parameters $\lambda$ and $\sigma$.  
More importantly, $\|\nabla f(x^+)\|$ behaves linearly with $\lambda\|s\|$, which will later be used to obtain the iteration-complexity of the proposed algorithm. 

\vspace{0.3cm}

Before deriving global consequences, we next show that, under the same regime for $\lambda$, the update \eqref{inexact} yields a guaranteed decrease in the objective value.

\begin{lemma}\label{lem:2.2}\label{lem:mondecr}
Suppose that 
%$f$ is convex and that 
Assumptions~{\bf A} and~{\bf D} hold.
If \eqref{cond:lambda:1} holds for some $\sigma>0$, then
\begin{equation}\label{eq:monde}
    f(x^+)-f(x)
    \le
    -\left(\frac{6\sigma\sqrt{\sigma}-6\sigma\kappa_B-2Hg^{1-\alpha}\sqrt{\sigma}}
           {6\sigma\sqrt{\sigma}}
-\frac{1}{\zeta}     \right)\lambda\|s\|^2.
\end{equation}
As a consequence, if $\zeta>2$ and  $\sigma$ satisfies
\begin{equation}\label{cond1:sigma}
    \sigma \ge 
\left(
\frac{
\kappa_B + \sqrt{\kappa_B^2 + \frac{4H}{3}\, g^{\,1-\alpha}\left(\frac12 - \frac{1}{\zeta}\right)}
}{
1 - \frac{2}{\zeta}
}
\right)^2:=\hat\sigma_2(g),
\end{equation}
then
\begin{equation}\label{cond:monotone}
    f(x^+)-f(x) \le -\frac{\lambda}{2}\|s\|^2.
\end{equation}
\end{lemma}
\begin{proof}
We first apply the second inequality in Assumption~{\bf A} with $y=x^+$
(i.e., $s=x^+-x$) to obtain
\[
f(x^+)-f(x)
\le
\langle \nabla f(x), s\rangle
+\frac{1}{2}\langle \nabla^{2}f(x)s,s\rangle
+\frac{H}{3}\|s\|^{3}.
\]
Using \eqref{inexact} (equivalently, $\nabla f(x)=-(B+\lambda{\cal I})s+r$),
it follows from the previous inequality that
\begin{align*}
f(x^+)-f(x)
&\le
-\langle Bs,s\rangle-\lambda\|s\|^2
+\frac{1}{2}\langle \nabla^{2}f(x)s,s\rangle
+\frac{H}{3}\|s\|^{3} +\langle r,s\rangle\\
&=
-\lambda\|s\|^2-\frac{1}{2}\langle Bs,s\rangle
+\frac{1}{2}\langle (\nabla^{2}f(x)-B)s,s\rangle
+\frac{H}{3}\|s\|^{3}+\langle r,s\rangle\\
&\le
-\lambda\|s\|^2
+\frac{\kappa_B}{2}\sqrt{\|\nabla f(x)\|^\alpha}\,\|s\|^2
+\frac{1}{2}\langle (\nabla^{2}f(x)-B)s,s\rangle
+\frac{H}{3}\|s\|^{3}+\langle r,s\rangle,
\end{align*}
where the last inequality follows from \eqref{Rayleigh}.
Combining the previous inequality with \eqref{inexact}, \eqref{def:s},
\eqref{cond:lambda:1}, \eqref{eq:bousk1}, and Assumption~{\bf D}, we obtain
\begin{eqnarray}\label{mono:decre:pre}
f(x^+)-f(x)-\langle r,s\rangle
&\le&
-\lambda\|s\|^2
+\frac{\kappa_B}{2}\sqrt{\|\nabla f(x)\|^\alpha}\,\|s\|^2
+\frac{1}{2}\langle (\nabla^{2}f(x)-B)s,s\rangle
+\frac{H}{3}\|s\|^{3} \nonumber\\
&\le&
-\lambda\|s\|^2
+\frac{\kappa_B}{2\sqrt{\sigma}}\sqrt{\sigma\|\nabla f(x)\|^\alpha}\,\|s\|^2
+\frac{\kappa_B}{2\sqrt{\sigma}}\sqrt{\sigma\|\nabla f(x)\|^\alpha}\,\|s\|^2
+\frac{H}{3}\|s\|^{3} \nonumber\\
&\le&
-\lambda\|s\|^2
+\frac{\kappa_B}{\sqrt{\sigma}}\lambda\|s\|^2
+\frac{H}{3}\frac{\lambda^{\frac{2-\alpha}{\alpha}}}{\sigma^{\frac{1}{\alpha}}}\|s\|^{2} \nonumber\\
&\le&
-\lambda\|s\|^2
+\frac{\kappa_B}{\sqrt{\sigma}}\lambda\|s\|^2
+\frac{H}{3\sigma}g^{1-\alpha}\lambda\|s\|^{2}\nonumber \\
&=&
-\frac{6\sigma\sqrt{\sigma}-6\sigma\kappa_B-2Hg^{1-\alpha}\sqrt{\sigma}}
       {6\sigma\sqrt{\sigma}}
 \lambda\|s\|^2.
\label{eq:mon:decrs}
\end{eqnarray}
Since $
\langle r,s\rangle\le \|r\|\|s\|\le \theta\|s\|^2\le (\lambda/\zeta)\|s\|^2
% \quad\text{and}\quad 
% \|\nabla f(x)\|
% \le
% \frac{3}{2(1-\theta)}\lambda\|s\|,
$, in view of \eqref{inexact} and \eqref{cond:lambda:1},
we conclude that
% \[
% \langle r,s\rangle\le \frac{3\theta}{2(1-\theta)}\lambda\|s\|^2,
% \]
\eqref{eq:monde} follows from \eqref{mono:decre:pre} and the previous inequality.
Moreover, \eqref{cond:monotone} follows as a direct consequence of
\eqref{cond1:sigma} and \eqref{eq:monde}.
\end{proof}
%\newpage

The previous result establishes the key \emph{monotonicity property} of the regularized quasi-Newton step.  
For sufficiently large $\sigma$, it shows that the decrease is proportional to $\lambda\|s\|^2$, which plays a central role in the global complexity analysis of Algorithm~1. 

%The following result establishes a lower bound on $f(x)-f(x^+)$ for $x$ and $x^+$ as in \eqref{inexact}. 
\vspace{0.3cm}
We now study how the error condition \eqref{cond:matrix:B} can be enforced in practice.
The next lemma shows that a symmetric finite-difference Hessian approximation
satisfies Assumption~{\bf C} whenever the finite-difference parameter $h$ is chosen
sufficiently small.

\begin{lemma}\label{lem:2.3}
Suppose that Assumption~{\bf A} holds.
Let $A\in\mathbb{R}^{n\times n}$ be the forward finite-difference matrix
\begin{equation}\label{eq:2.14}
    A=
    \left[
        \frac{\nabla f(x+h e_1)-\nabla f(x)}{h},
        \ldots,
        \frac{\nabla f(x+h e_n)-\nabla f(x)}{h}
    \right].
\end{equation}
Define 
\begin{equation}\label{eq:2.16}
    B := \frac{1}{2}(A + A^\top).
\end{equation}
Then
\begin{equation}\label{eq:2.17mais}
    \|B - \nabla^2 f(x)\| \le \sqrt{n}\, H h.
\end{equation}
\end{lemma}

\begin{proof}
By Assumption~{\bf A}, applied with $y = x + h e_i$, we have
\[
\|\nabla f(x + h e_i) - \nabla f(x) - h \nabla^2 f(x) e_i\|
\;\le\;
H h^2.
\]
Dividing both sides of the previous inequality by $h>0$, it follows that
\[
\left\|
\frac{\nabla f(x + h e_i) - \nabla f(x)}{h}
- \nabla^2 f(x) e_i
\right\|
\;\le\;
H h.
\]
By the definition of the finite-difference Hessian approximation $A$, this implies
\[
\|(A - \nabla^2 f(x)) e_i\| \;\le\; H h,
\qquad i = 1, \ldots, n.
\]
Consequently,
\[
\|A - \nabla^2 f(x)\|^2
\;\le\;
\|A - \nabla^2 f(x)\|_F^2
=
\sum_{i=1}^n \|(A - \nabla^2 f(x)) e_i\|^2
\;\le\;
n H^2 h^2,
\]
which yields
\begin{equation}
\|A - \nabla^2 f(x)\|
\;\le\;
\sqrt{n}\, H h.
\label{eq:2.17}
\end{equation}
Finally, combining \eqref{eq:2.16} and \eqref{eq:2.17}, we obtain
\[
\|B - \nabla^2 f(x)\|
\;\le\;
\|A - \nabla^2 f(x)\|
\;\le\;
\sqrt{n}\, H h,
\]
which completes the proof.
\end{proof}

This lemma quantifies how accurately a symmetric finite-difference matrix approximates the Hessian.  
Combined with a specific choice of $h$ depending on $\|\nabla f(x)\|^\alpha$, it guarantees that Assumption~{\bf C} is satisfied and thus allows us to plug the  matrix as in \eqref{eq:2.14} into the entire descent analysis.

We now combine the previous three lemmas to obtain a complete sufficient condition ensuring both decrease and gradient control for the regularized quasi-Newton step.

\begin{theorem}\label{thm:2.1}
Suppose that Assumptions~{\bf A} and {\bf D} hold, $\alpha\in(0,1]$,  $\zeta>2$,  $\sigma>0$, and that $x^+$ is generated as in \eqref{inexact}.  
Assume further that $B$ and $A$ are as in \eqref{eq:2.16} and \eqref{eq:2.14}, respectively, and step size  $h$ satisfies
\begin{equation}\label{eq:2.18}
    0 < h \le 
    \frac{\kappa_B}{4\sqrt{n}\sigma}
    \sqrt{\|\nabla f(x)\|^\alpha}.
\end{equation}
If  $\lambda$ is as in \eqref{cond:lambda:1} and $\sigma$ satisfies the condition 
\begin{equation}\label{cond:geral:sigma}
\sigma \ \ge\ 
\max\left\{
4\kappa_B^2,\; 
\hat\sigma_1(g), \hat\sigma_2(g)
\right\},
\end{equation}
where $\hat\sigma_1(g)$ and $\hat\sigma_2(g)$ are as in  \eqref{cond:sigma:grad} and \eqref{cond1:sigma}, respectively, then the following inequality holds:
\begin{align}
    f(x) - f(x^+) 
    &\;\ge\; \frac{\lambda}{2}\|x^+ - x\|^2, 
    \label{eq:2.20}
    \\[3mm]
    \|\nabla f(x^+)\|
    &\le 2\lambda \|s\|.
    \label{eq:2.211}
\end{align}
\end{theorem}

\begin{proof}
Using \eqref{cond:geral:sigma} and \eqref{eq:2.18}, we obtain
\begin{equation*}
0<h<\dfrac{\kappa_{B}}{\sqrt{n}H}\sqrt{\|\nabla f(x)\|^\alpha}.
\end{equation*}
The previous inequality, together with Lemma~\ref{lem:2.3}, implies that
\begin{equation}
\|B-\nabla^{2}f(x)\|\leq\sqrt{n}Hh<\kappa_{B}\sqrt{\|\nabla f(x)\|^\alpha}.
\label{eq:2.2234}
\end{equation}
Hence, since $\sigma$ satisfies \eqref{cond:geral:sigma} and \eqref{eq:2.2234} holds,
we can apply Lemmas~\ref{lem:mondecr} and~\ref{lem:3.1} and conclude that
\eqref{eq:2.20} and \eqref{eq:2.211} hold.
\end{proof}

The previous theorem provides the essential theoretical backbone for the construction of the main algorithm in this paper.  
It ensures that once $\lambda$ and $h$ are selected according to the prescribed rules, the resulting quasi-Newton step satisfies:
1) a \emph{descent condition}, and  
2) a \emph{first-order progress condition}.  
The inequality in \eqref{eq:2.20} serves as the acceptance criteria in Step~1.2 of Algorithm~1.

\section{Analysis of Algorithm 1}
In this section, we present Algorithm~1. It adaptively selects the parameters 
$\lambda$ and $h$ to guaranty that the theoretical conditions derived in the previous section remain valid throughout its iterations.
%The algorithm increases the regularization parameter until both the descent inequality \eqref{eq:mondecr} and the gradient bound \eqref{eq:boundgra1} are satisfied.  
%This adaptive procedure eliminates the need for prior knowledge of constants such as $H$ or $\kappa_B$.

The algorithm is formally described below.

\begin{mdframed}
	\textbf{Algorithm 1.} 
	\\[0.2cm]
	\textbf{Step 0.} Given $x_{1}\in\mathbb{R}^{n},$   {$\sigma_{1}> 0$}, $\alpha\in (0, 1]$, $\theta\in (0, 1)$, $\zeta>2$,  choose $\kappa_B>0$, and set $k:=1$.
	\\
	{\bf Step 1.}  Find the smallest integer $i\geq 0$ such that $2^i\sigma_{k}\geq 2\sigma_{1}.$
	\\
	%{\bf Step 1.1.} Set $\lambda_k=\sqrt{2^j\sigma_{k}\|\nabla f(x_k)\|}$ and
	%\small
	%\\
{\bf Step 1.1.} Define
\begin{equation}
\lambda_{k,i}=\max\left\{ [2(1+\theta)]^{\frac{\alpha}{2}}\sqrt{2^i\sigma_{k} \|\nabla f(x_k)\|^\alpha}, \zeta\theta\right\}
, \quad h_{i}=\dfrac{\kappa_{B}\sqrt{\|\nabla f(x_k)\|^\alpha}}{4\sqrt{n}\left(2^{i}\sigma_{k}\right)},
\label{eq:3.0} 
\end{equation}
compute
\begin{equation}
A_{k,i}=\left[\dfrac{\nabla f(x_{k}+h_ie_{1})-\nabla f(x_{k})}{h_i},\ldots,\dfrac{\nabla f(x_{k}+h_ie_{n})-\nabla f(x_{k})}{h_i}\right]\in\mathbb{R}^{n\times n},
\label{eq:3.1}
\end{equation}
define
\begin{equation}
B_{k,i}=\dfrac{1}{2}\left(A_{k,i}+A_{k,i}^{T}\right)
\label{eq:3.2}
\end{equation}
and compute $x^+$ satisfying
	\begin{equation}
		\|(B_{k,i}+\lambda_{k,i}{\cal I})(x^+-x_k)+\nabla f(x_k)\|\le \theta\min\left\{\|\nabla f(x_k)\|, \|x^+-x_k\|\right\}
		\label{eq:3.4}
	\end{equation}
	\normalsize
	\\
	{\bf Step 1.2.} If 
	\begin{equation}
		f(x^+)\leq f(x_k)- \dfrac{\lambda_{k,i}}{2}\|x^+-x_k\|^2
		\label{eq:mondecr}
	\end{equation}
and
\begin{equation}
\|\nabla f(x^+)\|\leq 2\lambda_{k,i}\|x^+-x_k\| 
\label{eq:boundgra1}
\end{equation}
hold, set $i_k=i$ and go to Step 2. Otherwise, set $i:=i+1$ and go to Step 1.1.
	\\
	{\bf Step 2.} Set $x_{k+1}=x^+,$ $\sigma_{k+1}=2^{i_k-1}\sigma_{k},$ $k:=k+1$ and go to Step 1.%, $\sigma_{t+1}=2^{i_{t}-1}\sigma_{t}$, $t:=t+1$, .
\end{mdframed}

We now make some remarks about Algorithm~1.
First, it is a parameter-free algorithm.
Second, since the matrix $B_{k,i}$ is computed as in \eqref{eq:3.2}
(with $B_{k,i}$ defined as in \eqref{eq:3.1}),
Algorithm~1 can be regarded as a Hessian-free Newton-type method.
Third, the parameter $\sigma_k$ controls the scale of both the regularization
and the finite-difference step.
The algorithm doubles $\sigma_k$ until the theoretical guarantees of
Theorem~\ref{thm:2.1} are satisfied, thereby ensuring that each iteration
is well defined without requiring user-chosen constants.

\vspace{0.3cm}

We next show that the monotone decrease enforced by Algorithm~1 implies that all iterates remain in the initial sublevel set, which in turn yields a uniform bound on the gradient.
%To establish the global
%convergence rate of Algorithm 1, 
First, we impose the following standard
boundedness assumption on the level set generated by the algorithm.
Such an assumption is common in the literature on cubic-regularized Newton
methods and related second-order optimization schemes.\\
\noindent
% \textbf{Assumption D:} Suppose that there exists $x^* \in \mathbb{R}^n$ such that $f(x^*):=\min_{x\in \mathbb{R}^n} f(x)$. Moreover,  the set $W:=\{x \in \mathbb{R}^n: f(x) \leq f(x_1)\}$ is bounded, that is, there exists $D>0$ constant such that $\|x-x^*\| \leq D$ for any $x \in W$. 

% \begin{lemma}[Bounded gradient on the initial sublevel set]
% Suppose that $f:\mathbb{R}^n\to\mathbb{R}$ is continuously differentiable and that
% Assumption~{\bf D} holds. 
% Let $\{x_k\}$ be the sequence generated by Algorithm~1.
% Then, $x_k\in W:=\{x\in\mathbb{R}^n:\ f(x)\le f(x_1)\}$ for all $k\ge 1$, and there exists
% a constant $M>0$ such that
% \begin{equation}\label{lem:grad-bounded}
% \|\nabla f(x)\|\le M \qquad \forall x\in W.
% \end{equation}
% In particular, $\|\nabla f(x_k)\|\le M$ for all $k\ge 1$.
% \end{lemma}

% \begin{proof}
% By the acceptance condition \eqref{eq:mondecr} in Algorithm~1, every accepted step satisfies
% \[
% f(x_{k+1}) \le f(x_k) - \frac{\lambda_{k,i_k}}{2}\|x_{k+1}-x_k\|^2 \le f(x_k),
% \]
% and hence $f(x_k)\le f(x_1)$ for all $k\ge 0$. Therefore, $x_k\in W$ for all $k$.
% Since $f$ is continuous, the set $W$ is closed. Moreover, by Assumption~{\bf D}, $W$ is bounded.
% Thus, $W$ is compact. Because $f\in C^1$, the mapping $x\mapsto \|\nabla f(x)\|$ is continuous,
% and hence
% \[
% M:=\max_{x\in W}\|\nabla f(x)\|<\infty.
% \]
% This proves the claim.
% \end{proof}

Before proving the main result of this section, we first show that the regularization sequence $\{\sigma_k\}$ that appears in Algorithm 1 is uniformly bounded above and below.

\begin{corollary}\label{cor:gooddef}
Under  Assumptions {\bf A}, {\bf C} and {\bf D}, the sequence $\{\sigma_k\}$ generated by Algorithm~1 satisfies
\begin{equation}\label{eq:bound.sigmak}
		\sigma_{1}\leq\sigma_{k}\leq 2\max\left\{
4\kappa_B^2, \hat\sigma_1(M), \hat\sigma_2(M)
\right\} +\sigma_{1} \coloneqq \sigma_{\max}, \quad k=1,\ldots,
	\end{equation}
where $M$ is as in \eqref{M:grad-bounded}.
\end{corollary}

\begin{proof}
We first observe that since $\alpha\in (0, 1]$ we have that the function $\phi(t)=t^{1-\alpha}$, $t>0$, is non-decreasing monotone.  Hence, it follows from \eqref{M:grad-bounded}, \eqref{cond:sigma:grad}, \eqref{cond1:sigma} that 
\[
\hat\sigma_1(g)\le \hat\sigma_1(M), \quad \hat\sigma_2(g)\leq\hat\sigma_2(M).
\]

We now show that \eqref{eq:bound.sigmak} holds. 
Clearly, \eqref{eq:bound.sigmak} is true for $k =1$, and thus our induction base holds.
 Suppose that \eqref{eq:bound.sigmak} holds for some $k \geq 1$. 
 If $i_{k}=0$, then by Steps 1 and 2, and the induction hypothesis, we have
	\begin{equation*}
		\sigma_{1}\leq\sigma_{k +1}= 2^{-1}\sigma_{k}\leq\sigma_{k}\leq \sigma_{\max}, 
	\end{equation*}
	that is, \eqref{eq:bound.sigmak} holds for $k+1$. 
    On the other hand, if $i_{k}\geq 1,$ then we claim that
	\begin{equation}\label{eq:ineq.ins.cor.gd}
		2^{i_{k}-1}\sigma_{k} \leq \sigma_{\text{max}}.
	\end{equation}
	Indeed, assuming by contradiction that \eqref{eq:ineq.ins.cor.gd} is not true, it follows from \eqref{eq:bound.sigmak} that
\begin{align*}
2^{i_k-2}\sigma_k
&\overset{\eqref{eq:ineq.ins.cor.gd}}{>}
2^{-1}\sigma_{\max} \\
&\overset{\eqref{eq:bound.sigmak}}{=}
\max\left\{
4\kappa_B^2, \hat\sigma_1(M), \hat\sigma_2(M)
\right\} + \frac{\sigma_1}{2} \\
&>
\max\left\{
4\kappa_B^2, \hat\sigma_1(g), \hat\sigma_2(g)
\right\}.
\end{align*}
In this case,
    it follows from Theorem~\ref{thm:2.1} that inequality \eqref{eq:mondecr} would have been satisfied for $i=i_{k}-1$, contradicting the minimality of $i_{k}$. 
    Thus, \eqref{eq:ineq.ins.cor.gd} is true. Consequently, by Steps 1 and 2, and \eqref{eq:ineq.ins.cor.gd}, we have
	\begin{equation*}
		\sigma_{1}\leq \sigma_{k+1}= 2^{i_{k}-1}\sigma_{k} \leq \sigma_{\text{max}},
	\end{equation*}
	that is, \eqref{eq:bound.sigmak} also holds for $k+1$ in this case. This completes the induction argument.
\end{proof}

The previous result formalizes that the adaptive strategy never drives $\sigma_k$ to infinity: the sequence remains well-scaled and stable.  
This fact is crucial for obtaining a global iteration-complexity bound.

We conclude this section by demonstrating that
% , under Assumptions {\bf A}, {\bf B} and {\bf C},
the iteration-complexity bound for Algorithm 1 is  $\mathcal{O}\left(\epsilon^{-2}\right)$ for some $\epsilon>0$ given. 
\begin{theorem}
	\label{thm:3.1}
Suppose that Assumptions {\bf A} to {\bf D} hold. 
Let $\{x_k\}$ be generated by Algorithm~1 and $\epsilon>0$ be given. 
Suppose that
\begin{equation}\label{eq:complex1}
    \|\nabla f(x_k)\| > \epsilon,
    \qquad
    k = 1,\ldots,T.
\end{equation}
Then
\begin{equation*}%\label{comple:T}
    T < 1+
    \frac{
        18 \sqrt{2\sigma_{\max} M^\alpha}
        \big( f(x_1) - f^{\mathrm{low}} \big)
    }{\epsilon^2}.
\end{equation*}
\end{theorem}

\begin{proof}
%It follows from 
Let $k$ be an index generated by Algorithm 1. 
Inequalities  \eqref{eq:mondecr}, \eqref{eq:boundgra1}, and Step 2 of  Algorithm 1 imply that 
\begin{align*}
    f(x_k)- f(x_{k+1})\ge \dfrac{\lambda_{k,i_k}}{2}\|x_{k+1}-x_k\|^2\geq \dfrac{\lambda_{k,i_k}}{2}\frac{\|\nabla f(x_{k+1})\|^2}{2(\lambda_{k,i_k})^2}=\frac{\|\nabla f(x_{k+1})\|^2}{4\lambda_{k,i_k}}, \quad \forall k=1,2, \ldots,
\end{align*}
 On the other hand, we observe that it is a direct consequence of \eqref{eq:3.0}, \eqref{M:grad-bounded}, and the fact that $2\sigma_{k+1} =2^{i_k}\sigma_{k}$, due to Step 2 of Algorithm 1, that 
\[
\lambda_{k,i_k}=\sqrt{2^{i_k}\sigma_{k}\|\nabla f(x_k)\|^\alpha}=\sqrt{2\sigma_{k+1}\|\nabla f(x_k)\|^\alpha}\le \sqrt{2\sigma_{\text{max}}M^\alpha},
\]
where $\sigma_{\text{max}}$ is as in \eqref{eq:bound.sigmak}.
Hence, combining the two previous inequalities and using \eqref{eq:complex1}, we obtain that
\begin{eqnarray*}
	 f(x_{1})-f^{\text{low}}&\ge&f(x_{1})-f(x_{T})
	=   \sum_{k=1}^{T-1}[f(x_{k})-f(x_{k+1})]\nonumber\\
	&\geq &\sum_{k=1}^{T-1} \frac{\|\nabla f(x_{k+1})\|^2}{18\lambda_{k,i_k}}> \frac{\epsilon^2}{18\sqrt{2\sigma_{\text{max}}M^\alpha}}(T-1).
\end{eqnarray*}
As a consequence, we have
\[
T<1+\dfrac{18\sqrt{2\sigma_{\text{max}}M^\alpha}(f(x_{1})-f^\text{low})}{\epsilon^2},
\]
which concludes the proof.
\end{proof}

The above theorem shows that Algorithm~1 achieves the worst-case evaluation complexity $
\mathcal{O}(\epsilon^{-2})
$
for driving the gradient norm below $\epsilon$.
This complexity %constitutes an improvement compared to the classical analysis where $\alpha=1$.
%Finally, it
matches the optimal worst-case rate for first-order methods but is obtained here for a regularized quasi-Newton scheme with finite-difference Hessian approximations and fully adaptive parameter selection.

\section{Global and Local Convergence Analysis}

In this section, we present the global and local convergence analysis of Algorithm 1. 

\subsection{The Global Convergence}

The goal of this subsection is to establish the global convergence properties of
Algorithm~1.
% The analysis is based on two main ingredients.
% First, we study a generic nonlinear recursion that captures the decay behavior
% of certain nonnegative sequences, which will be instrumental in deriving
% explicit convergence rates.
% Second, under standard assumptions on the objective function and the level
% set generated by the algorithm, we combine this recursion with the monotonicity
% and gradient-related properties of the iterates to obtain a global rate of
% convergence for the objective function values.

We begin by presenting a technical result that characterizes the asymptotic
behavior of sequences satisfying a specific descent inequality.
This result will later be applied to an appropriate subsequence of function
values generated by Algorithm~1.

\begin{proposition}\label{Prod:seq:alpha:t}
Let $\{\alpha_t\}_{t\ge 0}$ be a sequence of positive real numbers satisfying
\begin{equation}\label{cond:seq:alphat}
\alpha_{t+1} \le \alpha_t - C \alpha_t^{\frac{\alpha-4}{2}},
\end{equation}
for some constants $C>0$ and $\alpha \in (0,1]$.
Then,
\begin{equation}\label{order:alphat}
\alpha_t \le
\left(\frac{2}{C}\right)^{\frac{2}{2-\alpha}}\,
(t+1)^{-\frac{2}{2-\alpha}}.
%\mathcal{O}\!\left(t^{-\frac{2}{2-\alpha}}\right).
\end{equation}
\end{proposition}\

% \begin{proof}
%     Note that we have:

%     \begin{enumerate}
%         \item[i)] By hypothesis, $\alpha_{t+1} \leq \alpha_{t}$ for any $t$.
%         \item[ii)] Moreover, we get

%         \begin{equation*}
%             0 \leq \alpha_t - C \alpha_t^{2-\alpha/2} \leq \alpha_t \overset{\alpha_t \not = 0} {\implies} \alpha_t^{1-\alpha/2} \leq \frac{1}{C}. 
%         \end{equation*}
%     \end{enumerate}

%     On the other hand, we obtain that

%     \begin{equation*}
%         \frac{1}{\sqrt{\alpha_{t+1}}} - \frac{1}{\sqrt{\alpha_{t}}} \geq \frac{\alpha_t - \alpha_{t+1}}{2\alpha^{3/2}} \geq \frac{C \alpha_t^{2-\alpha/2}}{2 \alpha_t^{3/2}} \geq \frac{C}{2} \alpha_t^{\frac{1-\alpha}{2}}
%     \end{equation*}

%     So, 

%     \begin{equation*}
%         \frac{1}{\sqrt{\alpha_{t+1}}} \geq \frac{1}{\sqrt{\alpha_{t+1}}} - \frac{1}{\sqrt{\alpha_{0}}} = \sum_{k=0}^t \frac{1}{\sqrt{\alpha_{k+1}}} - \frac{1}{\sqrt{\alpha_{k}}} \geq \frac{C}{2} \sum_{k=0}^t \alpha_k^{\frac{1-\alpha}{2}} = \frac{C}{2} \alpha_t^{\frac{1-\alpha}{2}} (t+1) \geq \frac{C}{2} \alpha_{t+1}^{\frac{1-\alpha}{2}} (t+1) 
%     \end{equation*}

%     Therefore, by previous inequality, we conclude the proof. 
    
% \end{proof}

\begin{proof}
We first observe that the assumption in \eqref{cond:seq:alphat} implies that  $\{\alpha_t\}_{t\ge 0}$  is nonincreasing, and
$\alpha_t\le \alpha_0$ for all $t$.
Moreover, 
% since $\alpha_{t+1}\ge 0$, we have
% \[
% 0 \le \alpha_t - C\,\alpha_t^{\,2-\alpha/2},
% \]
% which implies (whenever $\alpha_t>0$) that
% \[
% \alpha_t^{\,1-\alpha/2} \le \frac{1}{C}.
% \]
using the mean value theorem applied to the function
$\phi(s)=s^{-1/2}$ on the interval $[\alpha_{t+1},\alpha_t]$, we obtain
\[
\frac{1}{\sqrt{\alpha_{t+1}}}-\frac{1}{\sqrt{\alpha_t}}
\ge
\frac{\alpha_t-\alpha_{t+1}}{2\,\alpha_t^{3/2}}.
\]
Combining this inequality with \eqref{cond:seq:alphat} yields
\[
\frac{1}{\sqrt{\alpha_{t+1}}}-\frac{1}{\sqrt{\alpha_t}}
\ge
\frac{C\,\alpha_t^{\,2-\alpha/2}}{2\,\alpha_t^{3/2}}
=
\frac{C}{2}\,\alpha_t^{\frac{1-\alpha}{2}}.
\]
Summing the above inequality from $k=0$ to $t$, we deduce that
\[
\frac{1}{\sqrt{\alpha_{t+1}}}
=
\frac{1}{\sqrt{\alpha_0}}
+
\sum_{k=0}^t
\left(
\frac{1}{\sqrt{\alpha_{k+1}}}
-
\frac{1}{\sqrt{\alpha_k}}
\right)
\ge
\frac{C}{2}\sum_{k=0}^t \alpha_k^{\frac{1-\alpha}{2}}.
\]
Since $\{\alpha_t\}$ is nonincreasing, we have
$\alpha_k \ge \alpha_t$ for all $k\le t$, and therefore
\[
\sum_{k=0}^t \alpha_k^{\frac{1-\alpha}{2}}
\ge
(t+1)\,\alpha_t^{\frac{1-\alpha}{2}}.
\]
The last two inequalities then imply that
\[
\frac{1}{\sqrt{\alpha_{t+1}}}
\ge
\frac{C}{2}(t+1)\,\alpha_t^{\frac{1-\alpha}{2}}
\ge
\frac{C}{2}(t+1)\,\alpha_{t+1}^{\frac{1-\alpha}{2}},
\]
which after simple calculations and using that $\alpha_{t+1}>0$ completes the proof.
% yields
% \[
% \alpha_{t+1}
% \le
% \left(\frac{2}{C}\right)^{\frac{2}{2-\alpha}}\,
% (t+1)^{-\frac{2}{2-\alpha}},
% \]
% which shows that \eqref{order:alphat} holds.
% \[
% \alpha_t = \mathcal{O}\!\left(t^{-\frac{2}{2-\alpha}}\right).
% \]
% This completes the proof.
\end{proof}

% To derive convergence rate results in the nonconvex setting, we assume that
% $f$ satisfies the Polyak--Łojasiewicz (PL) condition. This assumption,
% which is weaker than strong convexity and does not require convexity,
% ensures that the function value gap can be controlled by the squared
% norm of the gradient.

% \paragraph{Assumption E.}
% The function $f : \mathbb{R}^n \to \mathbb{R}$ is differentiable and satisfies the
% Polyak--Łojasiewicz (PL) condition; that is, there exists a constant $\zeta > 0$
% such that
% \[
% \frac{1}{2}\|\nabla f(x)\|^2
% \;\ge\;
% \zeta \big( f(x) - f(x^*) \big),
% \quad \forall x \in \mathbb{R}^n,
% \]
% where $x^*$ denotes a global minimizer of $f$.

We now turn to the global convergence analysis for Algorithm~1.

\begin{theorem}
Suppose that  Assumptions A to D hold, and that
% that $f$ is convex and 
$\{x_k\}$ is  generated by Algorithm 1. 
Moreover, assume that $f$ has a (finite) optimum $x^*$ such that $f(x^*)=\min_{x\in \mathbb{R}^n} f(x) $
and that 
%the Polyak--Łojasiewicz (PL) condition holds, that is, 
there exists a constant $\chi > 0$
such that
\begin{equation}\label{PL:condition}
\|x_k-x^*\|\le \chi.
\end{equation}
Then, 
    \begin{equation*}
        f(x_k) - f(x^*) = \mathcal{O}\left(k^{-\frac{2}{2-\alpha}}\right).
    \end{equation*}
\end{theorem}

\begin{proof}
We first observe that since $\{x_k\}$ is generated by Algorithm 1, then the sequence of functional values $\{f(x_k)\}_k$ is nonincreasing, due to \eqref{eq:mondecr} and Step 2 of Algorithm 1. 
As consequence,  $x_k \in {\cal L}(x_1)$ for all $k$. 
Using that $f$ is convex, \eqref{PL:condition} and \eqref{M:grad-bounded}, we have that for all $k$
\begin{equation}\label{def:D}
 f(x_k) - f(x^*) \leq \langle \nabla f(x_k),x^k - x^* \rangle \leq \| \nabla f(x_k)\| \|x^k - x^*\| \leq \chi \| \nabla f(x_k)\|.
\end{equation}
%where $D=M/(2\chi)$.
We now define the following two auxiliary sets:
    \begin{equation}\label{def:Jg}
        J_{g}:= \left\{i \in \mathbb{N}~:~ \| \nabla f(x_{i+1})\| \geq \frac{1}{6} \|\nabla f(x_i)\|\right\} \quad \text{and} \quad J_k=J_g\cap\{1,2,\dots,k\}.
    \end{equation}
Let $k \in J_g$ and denote $s_k := \|x_{k+1}-x_k\|$. 
Then, using \eqref{def:Jg}, \eqref{eq:3.0},  \eqref{eq:boundgra1}, \eqref{eq:bound.sigmak}, we obtain 
    \begin{equation*}
        \frac{1}{6} \|\nabla f(x_k)\| \leq \|\nabla f(x_{k+1})\| \leq 2 \lambda_{k,i_k} s_k \leq 2 \sqrt{\sigma_{\max}}\|\nabla f(x_k)\|^{\alpha/2}s_k,
    \end{equation*}
which implies that
\[
s_k \geq \frac{\| \nabla f(x_k)\|^{\frac{2-\alpha}{2}}}{12 \sqrt{\sigma_{\max}}}.
\]
Moreover, by the construction of $\lambda_{k,i_k}$ and \eqref{eq:bound.sigmak}, we have the lower bound
\[
\lambda_{k,i_k}
\ge\sqrt{2^{i_k}\sigma_k\,\|\nabla f(x_k)\|^\alpha}
\;\ge\;\sqrt{2\sigma_1}\,\|\nabla f(x_k)\|^{\alpha/2}.
\]
Then, it follows from \eqref{eq:mondecr}, Step 2 of Algorithm 1, \eqref{eq:3.0}, \eqref{eq:bound.sigmak}, \eqref{def:D}, and the previous inequality  that
\begin{align*}
f(x_{k+1}) - f(x_k)  \leq -\frac{1}{2} \lambda_{k,i_k} s_k^2 &\leq -\frac{1}{2} \sqrt{2\sigma_1}\| \nabla f(x_k)\|^{\frac{\alpha}{2}} \frac{\|\nabla f(x_k)\|^{2-\alpha}}{144\sigma_{\max}}\\
& = - \frac{\sqrt{2 \sigma_1}}{288 \ \sigma_{\max}} \| \nabla f(x_k)\|^{\frac{4-\alpha}{2}} \\
&\leq  - \frac{\sqrt{2 \sigma_1}}{288 \ \sigma_{\max} \chi^{\frac{4-\alpha}{2}}} (f(x_k)-f(x^*))^{\frac{4-\alpha}{2}}\\
&= - \beta (f(x_k)-f(x^*))^{\frac{4-\alpha}{2}},
\end{align*}
where $\beta := \frac{\sqrt{2 \sigma_1}\sigma_{\max}^{-1}}{288}\chi^{\frac{\alpha-4}{2}}$. 

Let the set $J_g$ be rewritten as $J_g := \{ i_1 < i_2< \cdots <i_t< \cdots\}$ and define the sequence $\alpha _t := \beta^2 (f(x_{i_t})-f(x^*) \geq 0$. Using that $i_{t+1} \geq i_t$, we obtain that
\begin{align*}
\alpha_{t+1}=\beta^2 (f(x_{i_{t+1}})-f(x^*) &\leq \beta^2 (f(x_{i_t+1})-f(x^*)) \\
&\leq \beta^2 (f(x_{i_t})-f(x^*)) - \beta^3 (f(x_{i_t})-f(x^*))^{{\frac{4-\alpha}{2}}}\\
&= \alpha_t - C \ \alpha_t^{{\frac{4-\alpha}{2}}},
\end{align*}
where $C := \beta^{\alpha-1}>0$. Hence, we conclude from Proposition \ref{Prod:seq:alpha:t} that $\alpha_t = \mathcal{O}(1/k^\frac{2}{2-\alpha})$. 

We now consider two cases. First,  we suppose that $m:=|J_k| \geq k/2$ and let $i_m \leq k$ be the largest element of $J_k$. Then, 
\begin{equation*}
    f(x_k) - f(x^*) \leq f(x_{i_m}) - f(x^*) = \mathcal{O}(1/|J_k|^{\frac{2}{2-\alpha}}). 
\end{equation*}
As a consequence, $f(x_k) - f(x^*) = \mathcal{O}(1/k^{\frac{2}{2-\alpha}}),$
because $|J_k| \geq k/2$.
In the second case, we assume that $|J_k| < k/2$. 
By Theorem \ref{thm:2.1}, we always have $\|\nabla f(x_{i+1})\| \leq 2 \|\nabla f(x_i)\|$ if $i \in J_k$, and  $\|\nabla f(x_{i+1})\| \leq \frac{1}{4} \|\nabla f(x_i)\|$ if $i\not \in J_k$. So, if $|J_k| \leq k/2$, we conclude that 
\begin{equation*}
    \frac{f(x_k)-f(x^*)}{\chi} \leq \| \nabla f(x_k)\| \leq \frac{2^{k/2}}{6^{k/2}}\| \nabla f(x_1)\| = \frac{\|\nabla f(x_1)\|}{3^{k/2}} = \mathcal{O}(1/k^{\frac{2}{2-\alpha}}).
\end{equation*}
Therefore, we conclude the proof. 
\end{proof}

\subsection{The Local Convergence Analysis}

In this section, we study the local convergence analysis of the Algorithm 1. 
For this purpose, we assume following additional assumption:\\
\noindent
{\bf Assumption E:} There exists $\mu>0$ such that 
$$
\nabla^{2}f(x_k)\succeq\mu {\cal I}, \quad \lambda_{k, i_k}= \max\left\{ [2(1+\theta)]^{\frac{\alpha}{2}}\sqrt{2^{i_k}\sigma_{k} \|\nabla f(x_k)\|^\alpha}, \zeta\theta, 2(1+\theta)\mu\right\},
$$ 
where $\{x_k\}$ is generated by Algorithm 1.

It is a direct consequence of Assumption {\bf E} and the first inequality in \eqref{eq:bousk1} that
\begin{equation}\label{eq:step-bound-mu}
\|s_k\|\leq
        \frac{2(1+\theta)\|\nabla f(x_k)\|}{\lambda_{k, i_k}}\le \frac{\|\nabla f(x_k)\|}{\mu},
    \end{equation}
where $s_k= x_{k+1} -x_k$.

Using Assumption {\bf E}, we can prove the following local superlinear convergence rate for Algorithm~1.
%%%%%%%%%%%%

\begin{theorem}
Assume that the previous assumptions hold, and let $\{x_k\}$ be generated by Algorithm~1. Then 
\begin{equation}\label{ineq:local-conv}
\| \nabla f(x_{k+1}) \| \leq \frac{2 \sqrt{2\sigma_{\max}}}{\mu} \ \|\nabla f(x_k)\|^{\frac{\alpha+2}{2}}.
    \end{equation}
Moreover, if the initial point $x_1$ satisfies the condition 
\begin{equation}\label{local-condition}
\| \nabla f(x_1)\| \leq \frac{1}{2}\left(\frac{\mu}{2 \sqrt{2\sigma_{\max}}}\right)^{\frac{\alpha}{2}},
    \end{equation}
then the sequence $\{\|\nabla f(x_k)\|\}_k$ converges superlinear to zero.
\end{theorem}

\begin{proof}
We first show that \eqref{ineq:local-conv} holds. 
It follows from Algorithm~1 (see \eqref{eq:boundgra1}) and  \eqref{eq:step-bound-mu} that the sequence $\{x_k\}$ satisfies 
\begin{equation*}
\|\nabla f(x_{k+1}) \| \leq 2 \lambda_{k, i_k} \|s_k\|\le \frac{2 \lambda_{k,i_k}\|\nabla f(x_k)\|}{\mu} \quad \forall \ k = 1, \ldots.
\end{equation*}
Thus, using \eqref{eq:bound.sigmak}, the definition of $\lambda_{k}$, and the previous inequality, we obtain
\begin{equation*}
\|\nabla f(x_{k+1}) \| \leq \frac{2 \sqrt{2\sigma_{\max}}}{\mu} \| \nabla f(x_k)\|^{\frac{\alpha+2}{2}} \quad \forall k = 1, \ldots,
\end{equation*}
which is \eqref{ineq:local-conv}.

We now show that \eqref{local-condition} holds.
By denoting 
\begin{equation*}
\delta_{k}=\left(\frac{2 \sqrt{2\sigma_{\max}}}{\mu}\right)^{\frac{2}{\alpha}}\|\nabla f(x_{k})\|,
\end{equation*}
it follows from \eqref{ineq:local-conv} that
\begin{align*}
 \delta_{k+1}=\left(\frac{2 \sqrt{2\sigma_{\max}}}{\mu}\right)^{\frac{2}{\alpha}}\|\nabla f(x_{k+1})\|&\le  \left(\frac{2 \sqrt{2\sigma_{\max}}}{\mu}\right)^{\frac{\alpha+2}{\alpha}}  \|\nabla f(x_{k})\|^{\frac{\alpha+2}{2}}.
 \end{align*}
 Since
 \begin{align*}
 \delta_k^{\frac{\alpha+2}{2}}&=\left[\left(\frac{2 \sqrt{2\sigma_{\max}}}{\mu}\right)^{\frac{2}{\alpha}}\|\nabla f(x_{k})\|\right]^{\frac{\alpha+2}{2}}  \\
 &= \left(\frac{2 \sqrt{2\sigma_{\max}}}{\mu}\right)^{\frac{\alpha+2}{\alpha}}\|\nabla f(x_{k})\|^{\frac{\alpha+2}{2}},
\end{align*}
we conclude that
\begin{equation*}
\delta_{k+1}\leq \delta_{k}^{\frac{\alpha+2}{2}}\quad\forall k\geq 1.
\end{equation*}
Moreover, by \eqref{local-condition}, we also have
\begin{equation*}
\delta_{1}=\left(\frac{2 \sqrt{2\sigma_{\max}}}{\mu}\right)^{\frac{2}{\alpha}}\|\nabla f(x_{1})\|\leq \dfrac{1}{2}.
\end{equation*}
Therefore, for all  $k\geq 2$,
\begin{align*}
\|\nabla f(x_{k})\|=\left(\frac{\mu}{2 \sqrt{2\sigma_{\max}}}\right)^{\frac{\alpha}{2}}\delta_{k}&\leq\left(\frac{\mu}{2 \sqrt{2\sigma_{\max}}}\right)^{\frac{\alpha}{2}}\delta_{1}^{\left(\frac{\alpha+2}{2}\right)^{k-1}}\\
&\leq \left(\frac{\mu}{2 \sqrt{2\sigma_{\max}}}\right)^{\frac{\alpha}{2}}\left(\frac{1}{2}\right)^{\left(\frac{\alpha+2}{2}\right)^{k-1}},
\end{align*}
and the proof is complete.
\end{proof}

\section{Numerical Experiments}

All experiments were conducted using Google Colab (Linux-based OS). The implementation was written in Python, and computations were carried out on a CPU with approximately 12 GB of available RAM. We consider two variants of Algorithm~1: One using finite-difference Hessian approximations, denoted by \texttt{AdN-FD\_inex}, and another using the exactly Hessian matrix, denoted by \texttt{AdN-H\_inex}, in the definition of $B_{k,i}$.  

Motivated by the observations in \cite{mishchenko2023regularized}, we choose the initial approximation for the Lipschitz constant as 
\begin{equation}\label{est:lipschitz:const}
H_0 := \frac{\|\nabla f(x_1)-\nabla f(x_0)-\nabla^2 f(x_0) (x_1 -x_0)\|}{\|x_1-x_0\|^2}.
\end{equation}
where $x_1,x_0 \in \mathbb{R}^n$, $x_0 \neq x_1$, are randomly generate initial points.
Using this estimate \eqref{est:lipschitz:const} and inequality \eqref{cond:geral:sigma}, the initial parameters for \texttt{AdN-FD\_inex} are set as $\kappa_B = 10^{-4}, \ g =  \|\nabla f(x_1)\|$ and 
$$ \sigma_1 = 
\max\left\{
4\kappa_B^2,\; 
\left(
\frac{\zeta}{2(\zeta-1)}
\left[
\kappa_B+\sqrt{\kappa_B^2+\frac{4(\zeta-1)}{\zeta}H_0 g^{1-\alpha}}
\right]
\right)^2, \left( \frac{
\kappa_B + \sqrt{\kappa_B^2 + \frac{4H_0}{3}\, g^{\,1-\alpha}\left(\frac12 - \frac{1}{\zeta}\right)}
}{
1 - \frac{2}{\zeta}
}
\right)^2\right\}.
$$
For \texttt{AdN-H\_inex}, the same parameters are used, with the only modification being $\kappa_B =0$. Based on preliminary numerical experiments, the parameters were set as follows: for Problem 1, $\alpha = 0.95$, $\zeta = 2.01$, and $\theta = \varepsilon$. For Problem 2 with the \texttt{mushrooms} dataset, $\alpha = 1$, $\zeta = 3$, and $\theta = \varepsilon$ were used, while for the \texttt{w8a} dataset, $\alpha = 1$, $\zeta = 3$, and $\theta = 10^{-8}$ were adopted.

To solve the inexact subproblem, we used the \textit{Conjugate Gradient method} to compute an approximate solution of
$$
\min_{s \in \mathbb{R}^n} \nabla f(x_k)^{\top}s 
+ \frac{1}{2} s^{\top} B_{i,k}s 
+ \frac{1}{2} \lambda_{i,k}\|s\|^2,
$$
where the initial point is $s_0 = 0$ (equivalently, $y_0 = x_k$ with $s = y - x_k$).
The stopping criterion is given by
$$
\|(B_{i,k} + \lambda_{i,k} I)s + \nabla f(x_k)\|
\leq \theta \min\left\{\|\nabla f(x_k)\|, \|s\|\right\},
$$
and compute $x^{+}=x_k+s$.
In addition, we consider the variants \texttt{AdN-H} and \texttt{AdN-FD}, 
which correspond to the exact versions of \texttt{AdN-H\_inex} and 
\texttt{AdN-FD\_inex}, respectively. These variants are obtained by setting 
$\theta = 0$, which enforces the exact solution of the subproblem. In this case, 
the search direction $s$ is computed by solving the linear system
$$
(B_{i,k} + \lambda_{i,k} I)s = - \nabla f(x_k).
$$
For comparison with the finite-difference version of Algorithm~1, namely \texttt{AdN-FD\_inex}, we consider the method \emph{`CNM with finite-difference hessian approximations'} proposed in \cite{grapiglia2022cubic}, denoted here as \texttt{CNM-FD}. This comparison is justified by the shared use of cubic regularization and finite-difference Hessian approximations. The parameters $\sigma_1$, $\gamma$, and $\bar\theta$ were selected based on preliminary numerical experiments for each problem. For the first problem, we used $\sigma_1 = H_0/2$,
$\gamma = 3.37$, and $\bar\theta = 2.23$. For the second problem, we set $\sigma_1 = H_0/10^{4}$. In addition, for the \texttt{mushrooms} dataset, we used $\gamma = 4.5$ and $\bar\theta = 2.7$, whereas for the \texttt{W8a} dataset we used $\gamma = 4.5$ and $\bar\theta = 0.1 \,$. To solve each inexact cubic subproblem in \texttt{CNM-FD}, i.e, compute an approximate solution $x_{t,i}^+$ to the subproblem
$$
\min_{y \in \mathbb{R}^n} M_{x_t, 2^i \sigma_t}(y) :=f(x_t) + \langle \nabla f(x_t), y - x_t \rangle 
+ \frac{1}{2} \langle B_{t,i}(y - x_t), y - x_t \rangle 
+ \frac{2^i \sigma_t}{6} \|y - x_t\|^3,
$$
such that
$$
M_{x_t, 2^i \sigma_t}(x_{t,i}^+) \leq f(x_t)
\quad \text{and} \quad
\|\nabla M_{x_t, 2^i \sigma_t}(x_{t,i}^+)\| 
\leq \bar \theta \min \left\{ \|x_{t,i}^+ - x_t\|^2, \|\nabla f(x_t)\| \right\},
$$
we employ a gradient descent method based on \cite{CarmonDuchi2016} , with the update rule
\begin{equation}
y_{k+1} = y_k - \alpha_k \nabla M_{x_t,2^i\sigma_t}(y_k), 
\quad
\alpha_k = \frac{1}{\|B_{t,i}\|_F + 2^i\sigma_t \|y_k - x_t\|}, 
\quad 
y_0 = x_t - \nabla f(x_t).
\end{equation}
The choice of the Frobenius norm $\|\cdot\|_F$ instead of the spectral norm $\|\cdot\|_2$ is justified by preliminary numerical experiments, which show that using $\|\cdot\|_F$ leads to a lower CPU time.
 
For comparison with Algorithm~1 using the exact Hessian, namely \texttt{AdN-H\_inex}, we consider the adaptive regularized Newton method \emph{`AdaN'} proposed in
\cite{mishchenko2023regularized}, here also denoted as \texttt{AdaN}, which also serves as the theoretical motivation for our approach. In this case, $H_0$ is used as the initial approximation of the Lipschitz constant.\\
{\bf Log-Sum-Exp:} In our first experiment, we consider a more ill-conditioned optimization problem 
\begin{equation}\label{prob:log-sum-exp}
    \min_{x \in \mathbb{R}^n} \beta \log \left( \sum_{i=1}^m \exp \left( \frac{a_i^{\top}x-b_i}{\beta} \right) \right),
\end{equation}
where $A =[a_1^{\top} \cdots a_m^{\top}]^{\top} \in \mathbb{R}^{m\times n}$ and $b = [b_1, \cdots, b_m]^{\top} \in \mathbb{R}^m$. This function is a smooth approximation of the maximum operator, satisfying
$$
\max_i(a_i^{\top}x-bi)  \leq \beta \log \left( \sum_{i=1}^m \exp \left( \frac{a_i^{\top}x-b_i}{\beta} \right) \right)\leq \max_i(a_i^{\top}x-b_i)+ \beta \log(m).
$$
To compare the methods, we employ performance profiles \cite{DolanMore2002}. The test problems are constructed by varying the parameter $\beta$ over the set
$$
\texttt{Betas}:=\left\{0.01+\eta \frac{0.5-0.01}{49} \;:\; \eta\in \{0, \cdots,49\}\right\}.
$$
That is, the test problem set is defined as
$$
\texttt{problems}:=\{ \text{the problem \eqref{prob:log-sum-exp} with $\beta \in$ \texttt{betas}}\}.
$$
We define the set of solvers with finite-difference Hessian approximation as
$$
\texttt{solvers} = \{\texttt{AdN-FD}, \texttt{CNM-FD}, \texttt{AdN-FD\_inex}\},
$$
which are used for comparison among methods based on finite differences.
Similarly, for methods using the exact Hessian matrix, we consider
$$
\texttt{solvers} = \{\texttt{AdaN}, \texttt{AdN-H}, \texttt{AdN-H\_inex}\}.
$$
Let $n_s$ denote the number of solvers and $n_p$ the number of test problems. For each pair $(s,p) \in \texttt{solvers} \times \texttt{problems}$, we compute 
$$
t_{p,s} = \text{computing time required to solve problem \textit{p} by solver \textit{s}}.
$$
Following \cite{dolan2002benchmarking}, the performance ratio is then defined by
$$
r_{p,s} = \frac{t_{p,s}}{\min\{t_{p,s}:s \in \texttt{problems}\}},
$$
By convention, if solver $s$ fails to solve problem $p$, a large value is assigned to $r_{p,s}$. 
In our experiments, a failure is declared when the stopping criterion $k>4000$ is reached, 
in which case we set $t_{p,s}:=10^{5}$, yielding a large performance ratio $r_{p,s}$.

The performance profile for solver \textit{s} is given by
$$
P(r_{p,s} \leq \tau:1\leq s\leq n_s):= \frac{\#\{p \in \texttt{problems}:r_{p,s}\leq \tau\}}{n_p} 
,$$
where the parameter $\tau$ varies over the set
$$
\texttt{taus}:= \left\{1+ \eta \frac{5-1}{49} \;:\; \eta\in \{0, \cdots,49\}\right\}.
$$
 The numerical comparison with \texttt{CNM-FD} is reported in Figure \ref{fig:prob:101} through performance profiles for dimension $n = 100$ and $m = 1000$ (left), $n=200$ and $m = 2000$ (middle), $n=300$ and $m = 3000$ (right). The stopping criterion for all methods was $\|\nabla f(x_k)\| < \varepsilon = 10^{-6}$ or a maximum number of global iterations $k > 4000$. 
\begin{figure}[H]
    \centering
    \includegraphics[scale=0.35]{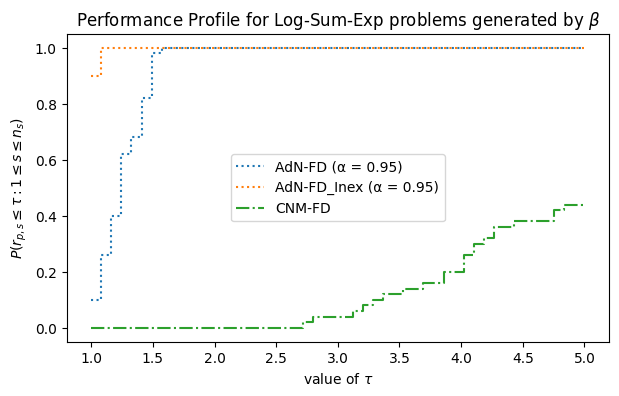}
    \includegraphics[scale=0.35]{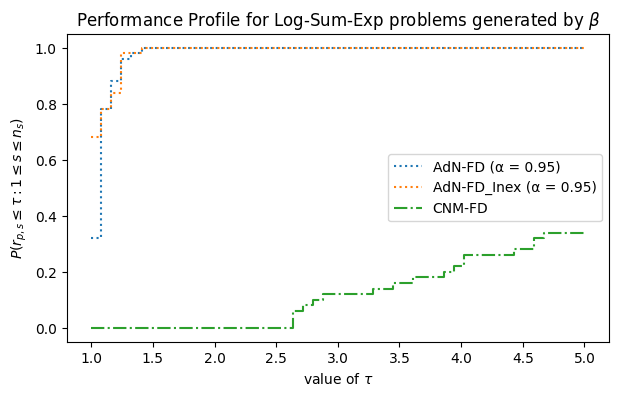}
    \includegraphics[scale=0.35]{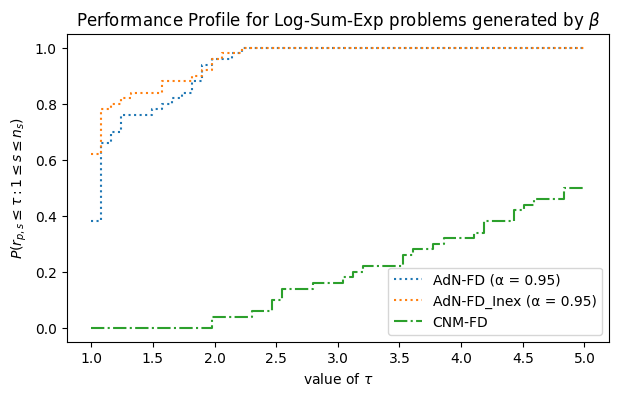}
    \caption{A comparison with CNM-FD using different problem sizes. }
    \label{fig:prob:101}
\end{figure}

In Figure \ref{fig:prob:101},  we observe that AdN variations were at least 10 times faster than \texttt{CNM-FD} most of the problems solved. 

Further, we show in Figure \ref{fig:prob:102},  using the same previous problem sizes and stopping criteria, the numerical results of the  comparison with \texttt{AdaN}.

\begin{figure}[h]
    \centering
    \includegraphics[scale=0.35]{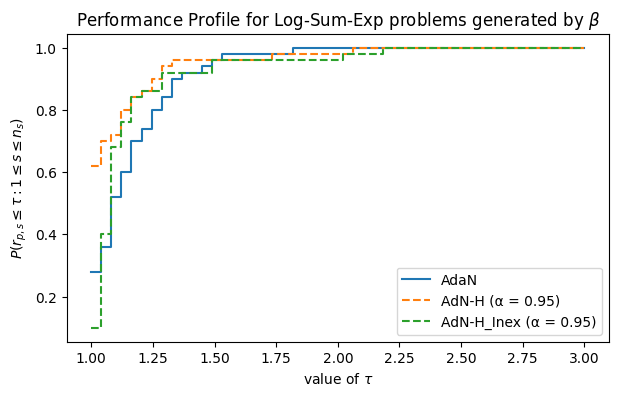}
    \includegraphics[scale=0.35]{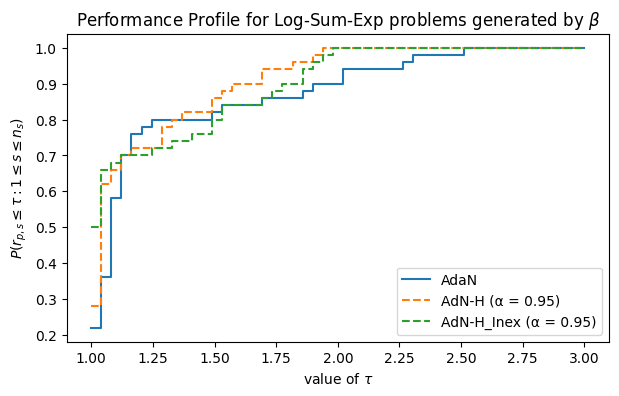}
    \includegraphics[scale=0.35]{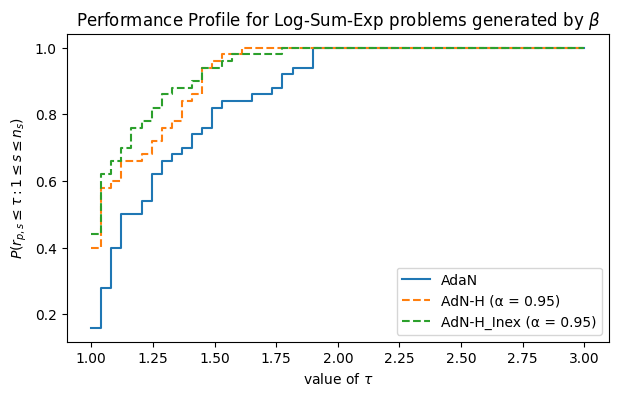}
    \caption{A comparison with \texttt{AdaN} using different problem sizes.}
    \label{fig:prob:102}
\end{figure}

Compared to \texttt{AdaN}, as shown in Figure \ref{fig:prob:102}, we note that the variations in AdN are better, but not as dominant as in relation to \texttt{CNM-FD} when considering CPU time.

\noindent\textbf{Logistic Regression.}
In our second experiment, we consider the $\ell_2$-regularized logistic regression problem
$$
\min_{x \in \mathbb{R}^n}
\frac{1}{m}\sum_{i=1}^m
\left[
- b_i \log(\sigma(a_i^{\top} x))
- (1-b_i)\log\bigl(1-\sigma(a_i^{\top} x)\bigr)
\right]
+ \frac{\ell}{2}\|x\|^2,
$$
where $\sigma:\mathbb{R}\to(0,1)$ denotes the logistic sigmoid function, $a_i \in \mathbb{R}^n$ is the feature vector of the $i$-th sample, and $b_i \in \{0,1\}$ is the corresponding binary label. We use the \texttt{w8a} and \texttt{mushrooms} datasets from the LIBSVM repository. The regularization parameter is set to $\ell = 10^{-10}$ in order to make the problem ill-conditioned.

Before presenting the numerical results for this problem, note that all five methods under consideration are adaptive. We therefore define $N_k$ as the total number of internal iterations performed up to the $k$-th global iteration, that is,
$$
N_k = i_0 + \cdots + i_k,
$$
where $i_k$ denotes the number of internal iterations at global iteration $k$.

The numerical comparison with CNM-FD and AdaN for Problem 2 on the \texttt{mushrooms} and \texttt{w8a} datasets is reported in Figure~\ref{fig:prob:2:1} and Figure~\ref{fig:prob:2:2}, respectively. The stopping criterion for all methods is either reaching a maximum number of global iterations ($k > 1000$) or satisfying $\|\nabla f(x_k)\| < \varepsilon$, with $\varepsilon = 10^{-6}$ for the \texttt{w8a} dataset and $\varepsilon = 10^{-11}$ for the \texttt{mushrooms} dataset. 

In addition, Tables ~\ref{tab1:problem:2} and \ref{tab2:problem:2} reports the CPU time (in seconds), the number of global iterations ($k$), the total number of internal iterations ($N_k$), and the norm of the gradient at the final point (denoted by $\|\nabla f(x^*)\|$).

\begin{figure}[h]
    \centering
    \includegraphics[scale=0.6]{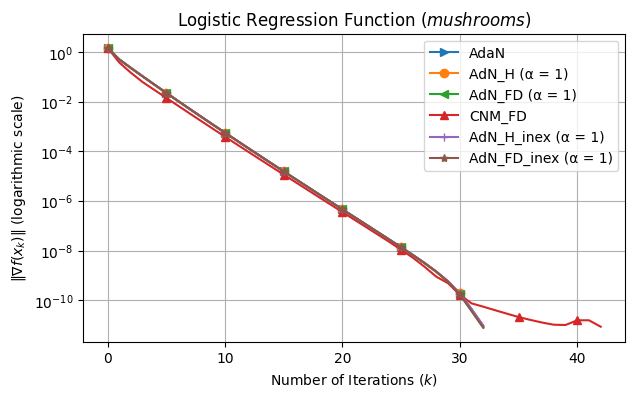}
    \caption{A comparison with \texttt{AdaN} and \texttt{CNM-FD} in the problem 2 with \texttt{mushrooms} dataset.}
    \label{fig:prob:2:1}
\end{figure}

\begin{table}[h]
\centering
\begin{tabular}{c c c c c}
\hline
Algorithms & CPU time (s) & Global iteration (k) &  Total iteration ($N_k$) & $\|\nabla f(x^*)\|$  \\
\hline
CNM\_DF & 8.86 & 42 & 71 & $8.64\times10^{-12}$\\
AdN\_DF & 3.93 & 32 & 42 & $7.87 \times10^{-12}$   \\
AdN\_DF\_inex & 3.55 & 32 & 40 & $7.88\times10^{-12}$   \\
\hline
AdaN & 0.65 & 32 & 42 & $9.31 \times10^{-12}$ \\
AdN-H & 0.66 & 32 & 42 & $9.31 \times10^{-12}$  \\
AdN-H\_inex & 0.62 & 32 & 42 & $9.32 \times10^{-12}$ \\

\hline
\end{tabular}
\caption{A comparison with \texttt{AdaN} and \texttt{CNM-FD} in the problem 2 with \texttt{mushrooms} dataset.}
\label{tab1:problem:2}
\end{table}

In Figure \ref{fig:prob:2:1} and Table \ref{tab1:problem:2}, we observe that all methods computed very similar iterates $x_k$, and the total number of iterations (k) was the same, except for \texttt{CNM-FD}. In terms of computational time, the inexact method performed better in both cases, i.e., when using finite differences (FD) and the exact Hessian.

\begin{figure}[h]
    \centering
    \includegraphics[scale=0.6]{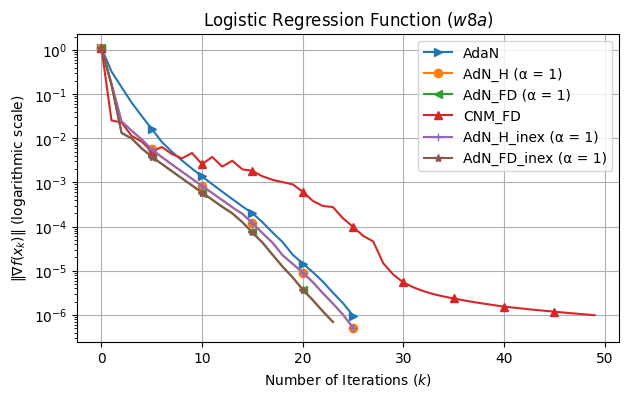}
    \caption{A comparison with \texttt{AdaN} and \texttt{CNM-FD} in the problem 2 with \texttt{mushrooms} and \texttt{w8a} datasets.}
    \label{fig:prob:2:2}
\end{figure}

\begin{table}[h]
\centering
\begin{tabular}{c c c c c }
\hline
Algorithms & CPU time (s) & Global iteration (k) &  Total iteration ($N_k$) & $\|\nabla f(x^*)\|$  \\
\hline

CNM\_DF & 194.96 & 49 & 107 & $9.87\times10^{-7}$  \\
AdN\_DF & 45.55 & 23 & 31 & $6.95 \times10^{-7}$   \\
AdN\_DF\_inex & 47.80 & 23 & 29 & $6.65 \times10^{-7}$  \\
\hline
AdaN & 2.38 & 25 & 35 & $9.63 \times10^{-7}$   \\
AdN-H & 2.25 & 25 & 34 & $5.13 \times10^{-7}$  \\
AdN-H\_inex & 2.30 & 25 & 34 & $5.13 \times10^{-7}$  \\

\hline
\end{tabular}
\caption{A comparison with \texttt{AdaN} and \texttt{CNM-FD} in the problem 2 with \texttt{w8a} dataset.}
\label{tab2:problem:2}
\end{table}

In Figure \ref{fig:prob:2:2} and Table \ref{tab2:problem:2}, we observe that the exact variant of Algorithm 1 computed iterates $x_k$
 that are very similar to those obtained by the inexact method, with a comparable total number of iterations (k). Moreover, in terms of the metric $iterations\times \|\nabla f\|$, Algorithm 1 and its variants performed better than the other methods. In terms of computational time, the exact method performed better in both cases, i.e., when using finite differences (FD) and the exact Hessian.

\clearpage

\section{Conclusion}
In this paper, we introduced a regularized Hessian-free inexact Newton-type method designed for smooth convex optimization. By combining a finite-difference approximation of the Hessian with an adaptive regularization mechanism, we developed a framework that eliminates the need for exact second-order derivatives while maintaining the strong theoretical guarantees of classical second-order methods.

Our theoretical analysis demonstrated that the proposed method achieves a global convergence rate of $\mathcal{O}(k^{-2})$ for convex functions with Lipschitz continuous Hessians, matching the optimal complexity bound for this class of problems. Furthermore, we established that the algorithm is robust to inexact subproblem solutions, which, coupled with the Hessian-free nature of the updates, makes it particularly well-suited for large-scale applications where computing or storing the exact Hessian is computationally prohibitive.

Numerical experiments confirmed the efficiency of our approach. The adaptive selection of the regularization parameter $\lambda_k$ and the precision of the finite-difference approximation allowed the method to outperform state-of-the-art regularized Newton schemes in terms of both iteration count and total CPU time. Specifically, the modified variant utilizing exact Hessians (when available) preserved local quadratic convergence, offering a seamless transition from global exploration to fast local refinement.

Future research directions include extending this framework to non-convex objectives and exploring the integration of stochastic Hessian-free approximations to address large-scale problems in machine learning and data science. Overall, the proposed method provides a sound and efficient alternative for high-performance convex optimization, bridging the gap between first-order simplicity and second-order speed.

\vspace{0.5em}
\noindent{\bf Acknowledgements}
This work was partially supported by grant CNPq (National Council for Scientific and Technological Development) No.~306593/2022-0 (Gilson N. Silva)  and by CNPq grant 309552/2022-2 (Paulo Sérgio M. Santos).

\vspace{0.5em}
\noindent{\bf Declarations}
\vspace{0.5em}

\noindent{\bf Conflict of interest}
The authors declare that they have no conflict of interest.
\bibliographystyle{plain} 
\bibliography{references}

@article{mishchenko2023regularized,
  title={Regularized Newton method with global convergence},
  author={Mishchenko, Konstantin},
  journal={SIAM Journal on Optimization},
  volume={33},
  number={3},
  pages={1440--1462},
  year={2023},
  publisher={SIAM}
}

@article{birgin2017use,
  title={The use of quadratic regularization with a cubic descent condition for unconstrained optimization},
  author={Birgin, Ernesto G and Mart{\'\i}nez, Jos{\'e} Mario},
  journal={SIAM Journal on Optimization},
  volume={27},
  number={2},
  pages={1049--1074},
  year={2017},
  publisher={SIAM}
}

@article{cartis2011adaptive,
  title={Adaptive cubic regularisation methods for unconstrained optimization. Part I: motivation, convergence and numerical results},
  author={Cartis, Coralia and Gould, Nicholas IM and Toint, Philippe L},
  journal={Mathematical Programming},
  volume={127},
  number={2},
  pages={245--295},
  year={2011},
  publisher={Springer}
}

@article{cartis2013evaluation,
  title={On the evaluation complexity of cubic regularization methods for potentially rank-deficient nonlinear least-squares problems and its relevance to constrained nonlinear optimization},
  author={Cartis, Coralia and Gould, Nicholas IM and Toint, Philippe L},
  journal={SIAM Journal on Optimization},
  volume={23},
  number={3},
  pages={1553--1574},
  year={2013},
  publisher={SIAM}
}

@article{cartis2019universal,
  title={Universal regularization methods: varying the power, the smoothness and the accuracy},
  author={Cartis, Coralia and Gould, Nick I and Toint, Philippe L},
  journal={SIAM Journal on Optimization},
  volume={29},
  number={1},
  pages={595--615},
  year={2019},
  publisher={SIAM}
}

@article{dan2002convergence,
  title={Convergence properties of the inexact Levenberg-Marquardt method under local error bound conditions},
  author={Dan, Hiroshige and Yamashita, Nobuo and Fukushima, Masao},
  journal={Optimization methods and software},
  volume={17},
  number={4},
  pages={605--626},
  year={2002},
  publisher={Taylor \& Francis}
}

@article{doikov2024super,
  title={Super-universal regularized Newton method},
  author={Doikov, Nikita and Mishchenko, Konstantin and Nesterov, Yurii},
  journal={SIAM Journal on Optimization},
  volume={34},
  number={1},
  pages={27--56},
  year={2024},
  publisher={SIAM}
}

@article{doikov2022high,
  title={High-order optimization methods for fully composite problems},
  author={Doikov, Nikita and Nesterov, Yurii},
  journal={SIAM Journal on Optimization},
  volume={32},
  number={3},
  pages={2402--2427},
  year={2022},
  publisher={SIAM}
}

@article{fan2005quadratic,
  title={On the quadratic convergence of the Levenberg-Marquardt method without nonsingularity assumption},
  author={Fan, Jin-yan and Yuan, Ya-xiang},
  journal={Computing},
  volume={74},
  number={1},
  pages={23--39},
  year={2005},
  publisher={Springer}
}

@techreport{griewank1981modification,
  title={The modification of Newton’s method for unconstrained optimization by bounding cubic terms},
  author={Griewank, Andreas},
  year={1981},
  institution={Technical report NA/12}
}

@article{grippo1986nonmonotone,
  title={A nonmonotone line search technique for Newton’s method},
  author={Grippo, Luigi and Lampariello, Francesco and Lucidi, Stefano},
  journal={SIAM journal on Numerical Analysis},
  volume={23},
  number={4},
  pages={707--716},
  year={1986},
  publisher={SIAM}
}

@article{li2004regularized,
  title={Regularized Newton methods for convex minimization problems with singular solutions},
  author={Li, Dong-Hui and Fukushima, Masao and Qi, Liqun and Yamashita, Nobuo},
  journal={Computational optimization and applications},
  volume={28},
  number={2},
  pages={131--147},
  year={2004},
  publisher={Springer}
}

@inproceedings{martens2010deep,
  title={Deep learning via hessian-free optimization.},
  author={Martens, James and others},
  booktitle={Icml},
  volume={27},
  pages={735--742},
  year={2010}
}

@article{nesterov2008accelerating,
  title={Accelerating the cubic regularization of Newton’s method on convex problems},
  author={Nesterov, Yu},
  journal={Mathematical Programming},
  volume={112},
  number={1},
  pages={159--181},
  year={2008},
  publisher={Springer}
}

@article{nesterov2021superfast,
  title={Superfast second-order methods for unconstrained convex optimization},
  author={Nesterov, Yurii},
  journal={Journal of Optimization Theory and Applications},
  volume={191},
  number={1},
  pages={1--30},
  year={2021},
  publisher={Springer}
}

@article{nesterov2006cubic,
  title={Cubic regularization of Newton method and its global performance},
  author={Nesterov, Yurii and Polyak, Boris T},
  journal={Mathematical programming},
  volume={108},
  number={1},
  pages={177--205},
  year={2006},
  publisher={Springer}
}

@article{polyak2009regularized,
  title={Regularized Newton method for unconstrained convex optimization},
  author={Polyak, Roman A},
  journal={Mathematical programming},
  volume={120},
  number={1},
  pages={125--145},
  year={2009},
  publisher={Springer}
}

@article{ueda2014regularized,
  title={A regularized Newton method without line search for unconstrained optimization},
  author={Ueda, Kenji and Yamashita, Nobuo},
  journal={Computational Optimization and Applications},
  volume={59},
  number={1},
  pages={321--351},
  year={2014},
  publisher={Springer}
}

@article{alvarez2025first,
  title={A first-order regularized algorithm with complexity properties for the unconstrained and the convexly constrained low order-value optimization problem: GQ {\'A}lvarez et al.},
  author={{\'A}lvarez, GQ and Birgin, EG and Mart{\'\i}nez, Jos{\'e} M{\'a}rio},
  journal={Journal of Global Optimization},
  volume={93},
  number={1},
  pages={241--261},
  year={2025},
  publisher={Springer}
}

@article{grapiglia2024worst,
  title={Worst-case evaluation complexity of a derivative-free quadratic regularization method},
  author={Grapiglia, Geovani Nunes},
  journal={Optimization Letters},
  volume={18},
  number={1},
  pages={195--213},
  year={2024},
  publisher={Springer}
}

@article{grapiglia2022cubic,
  title={A cubic regularization of Newton’s method with finite difference Hessian approximations},
  author={Grapiglia, Geovani Nunes and Gon{\c{c}}alves, Max LN and Silva, GN},
  journal={Numerical Algorithms},
  volume={90},
  number={2},
  pages={607--630},
  year={2022},
  publisher={Springer}
}

@article{dolan2002benchmarking,
  title={Benchmarking optimization software with performance profiles},
  author={Dolan, Elizabeth D and Mor{\'e}, Jorge J},
  journal={Mathematical programming},
  volume={91},
  number={2},
  pages={201--213},
  year={2002},
  publisher={Springer}
}

@article{CarmonDuchi2016,
  author  = {Carmon, Yair and Duchi, John C.},
  title   = {Gradient Descent Efficiently Finds the Cubic-Regularized Nonconvex Newton Step},
  journal = {arXiv preprint arXiv:1612.00564},
  year    = {2016}
}

@book{nocedal2006numerical,
  title={Numerical optimization},
  author={Nocedal, Jorge and Wright, Stephen J},
  year={2006},
  publisher={Springer}
}

@book{conn2000trust,
  title={Trust region methods},
  author={Conn, Andrew R and Gould, Nicholas IM and Toint, Philippe L},
  year={2000},
  publisher={SIAM}
}

\end{document}